\documentclass[american,english]{article}
\usepackage[T1]{fontenc}
\usepackage[latin9]{inputenc}
\usepackage{array}
\usepackage{bm}
\usepackage{multirow}
\usepackage{amsmath}
\usepackage{amsthm}
\usepackage{amssymb}
\usepackage[all]{xy}

\makeatletter

\providecommand{\tabularnewline}{\\}

\numberwithin{table}{section}
\numberwithin{equation}{section}
\theoremstyle{plain}
\newtheorem{thm}{\protect\theoremname}[section]
\theoremstyle{definition}
\newtheorem{defn}[thm]{\protect\definitionname}
\theoremstyle{plain}
\newtheorem{prop}[thm]{\protect\propositionname}
\theoremstyle{plain}
\newtheorem{lem}[thm]{\protect\lemmaname}
\theoremstyle{remark}
\newtheorem{rem}[thm]{\protect\remarkname}
\theoremstyle{plain}
\newtheorem{cor}[thm]{\protect\corollaryname}
\theoremstyle{remark}
\newtheorem{claim}[thm]{\protect\claimname}

\makeatother

\usepackage{babel}
\addto\captionsamerican{\renewcommand{\claimname}{Claim}}
\addto\captionsamerican{\renewcommand{\corollaryname}{Corollary}}
\addto\captionsamerican{\renewcommand{\definitionname}{Definition}}
\addto\captionsamerican{\renewcommand{\lemmaname}{Lemma}}
\addto\captionsamerican{\renewcommand{\propositionname}{Proposition}}
\addto\captionsamerican{\renewcommand{\remarkname}{Remark}}
\addto\captionsamerican{\renewcommand{\theoremname}{Theorem}}
\addto\captionsenglish{\renewcommand{\claimname}{Claim}}
\addto\captionsenglish{\renewcommand{\corollaryname}{Corollary}}
\addto\captionsenglish{\renewcommand{\definitionname}{Definition}}
\addto\captionsenglish{\renewcommand{\lemmaname}{Lemma}}
\addto\captionsenglish{\renewcommand{\propositionname}{Proposition}}
\addto\captionsenglish{\renewcommand{\remarkname}{Remark}}
\addto\captionsenglish{\renewcommand{\theoremname}{Theorem}}
\providecommand{\claimname}{Claim}
\providecommand{\corollaryname}{Corollary}
\providecommand{\definitionname}{Definition}
\providecommand{\lemmaname}{Lemma}
\providecommand{\propositionname}{Proposition}
\providecommand{\remarkname}{Remark}
\providecommand{\theoremname}{Theorem}

\begin{document}
\title{\textbf{\Large{}Key Varieties for Prime $\mathbb{Q}$-Fano Threefolds
Related with $\mathbb{P}^{2}\times\mathbb{P}^{2}$-Fibrations. Part
II}}
\author{Hiromichi Takagi}
\maketitle
\begin{abstract}
We construct a $15$-dimensional affine variety $\Pi_{\mathbb{A}}^{15}$
with a ${\rm GL}_{2}$- and $(\mathbb{C}^{*})^{4}$-actions. We denote
by $\Pi_{\mathbb{A}}^{14}$ the affine variety obtained from $\Pi_{\mathbb{A}}^{15}$
by setting one specified variable to $1$ (we refer the precise definition
to Definition \ref{def:Pi} of the paper). Let $\Pi_{\mathbb{P}}^{13}$
be several weighted projectivizations of $\Pi_{\mathbb{A}}^{14}$,
and $\Pi_{\mathbb{P}}^{14}$ the weighted cone over $\Pi_{\mathbb{P}}^{13}$
with a weight one coordinate added. We show that $\Pi_{\mathbb{P}}^{13}$
or $\Pi_{\mathbb{P}}^{14}$ produce, as weighted complete intersections,
examples of prime $\mathbb{Q}$-Fano threefolds of codimension four
belonging to the eight classes No.308, 501, 512, 550, 577, 872, 878,
and 1766 of the graded ring database \cite{key-6}. The construction
of $\Pi_{\mathbb{A}}^{15}$ is based on a certain type of unprojection
and is inspired by R.Taylor's thesis \cite{key-3} submitted to University
of Warwick. We also show that a partial projectivization of $\Pi_{\mathbb{A}}^{15}$
has a $\mathbb{P}^{2}\times\mathbb{P}^{2}$-fibration over the affine
space $\mathbb{A}^{10}$. To show this, we introduce another $13$-dimensional
affine variety $H_{\mathbb{A}}^{13}$ whose product with an open subset
of $\mathbb{A}^{2}$ is isomorphic to a sextic cover of an open subset
of $\Pi_{\mathbb{A}}^{15}$.
\end{abstract}
\tableofcontents{}

\section{Introduction}

\subsection{Background}

In this paper, we work over $\mathbb{C}$, the complex number field. 

A complex projective variety is called a \textsl{$\mathbb{Q}$-Fano
variety} if it is a normal variety with only terminal singularities,
and its anti-canonical divisor is ample. A $\mathbb{Q}$-Fano variety
is called \textit{prime} if its anti-canonical divisor generates the
group of numerical equivalence classes of divisors. This paper is
the sequel to \cite{key-5}, where we construct affine varieties $\Sigma_{\mathbb{A}}^{14}$
and $\Sigma_{\mathbb{A}}^{13}$, and produce, as weighted complete
intersections of their weighted projectivizations, several examples
of prime $\mathbb{Q}$-Fano threefolds of anti-canonical codimension
four. Here, by the anti-canonical codimension of a $\mathbb{Q}$-Fano
threefold $X$, we mean the codimension of $X$ in the weighted projective
space of the minimal dimension determined by the anti-canonical graded
ring of $X$, and we denote it by ${\rm ac}_{X}$.

In this paper, we construct a new affine variety $\Pi_{\mathbb{A}}^{14}$
and produce, as weighted complete intersections of its weighted projectivizations
$\Pi_{\mathbb{P}}^{13}$ or the weighted cones over $\Pi_{\mathbb{P}}^{13}$,
several new examples of prime $\mathbb{Q}$-Fano threefolds of anti-canonical
codimension four.

\subsection{Affine key varieties $\Pi_{\mathbb{A}}^{15}$ and $\Pi_{\mathbb{A}}^{14}$\label{subsec:Affine-key-varieties}}

To construct the affine variety $\Pi_{\mathbb{A}}^{14}$, we introduce
an affine variety $\Pi_{\mathbb{A}}^{15}$ wihich will contain $\Pi_{\mathbb{A}}^{14}$
as a divisor. For this, we start from a non-normal affine variety
whose normalization is $\mathbb{A}^{6}$, whose construction is a
modification of those given in \cite[Ex.~5.2.2]{key-3}. Let $\mathbb{A}^{6}\rightarrow\mathbb{A}^{8}$
be the morphism defined by {\small{}$(p_{1},p_{2},p_{3},p_{4},t_{1},w)\mapsto(p_{1},p_{2},p_{3},p_{4},u,v,t_{1},t_{2})$}
with 
\begin{equation}
u=wp_{1}+w^{2}p_{2},v=wp_{3}+w^{2}p_{4},t_{2}=w^{3}+t_{1}w\label{eq:zb}
\end{equation}
and $p_{1},p_{2},p_{3},p_{4},t_{1}$ being unchanged. We can easily
check that this is finite and birational onto its image, hence the
image is non-normal. Moreover, we can show that the image is defined
by all the $3\times3$ minors of the matrix

{\small{}
\begin{equation}
\mathsf{M}=\left(\begin{array}{cccccc}
u & v & -t_{2}p_{2} & -t_{2}p_{4} & -t_{2}p_{1}+t_{1}u & -t_{2}p_{3}+t_{1}v\\
-p_{1} & -p_{3} & u+t_{1}p_{2} & v+t_{1}p_{4} & -t_{2}p_{2} & -t_{2}p_{4}\\
-p_{2} & -p_{4} & -p_{1} & -p_{3} & u & v
\end{array}\right)\label{eq:M}
\end{equation}
\noindent} (see Proposition \ref{prop:EqD}). For $i,j,k$ with $1\leq i<j<k\leq6,$
we denote by $\mathsf{M}_{ijk}$ the $3\times3$ matrix whose $1,2,3$
th columns coincide with $i,j,k$ th columns of $\mathsf{M}$, respectively
and set $D_{ijk}:=\det\mathsf{M}_{ijk}$. It is easy to see that the
ideal generated by all the $3\times3$ minors of $\mathsf{M}$ in
$\mathbb{C}[p_{1},p_{2},p_{3},p_{4},u,v,t_{1},t_{2}]$ are generated
by $D_{123},D_{124},D_{125},D_{126},D_{135},D_{136}':=D_{136}+2D_{145},D_{245}':=D_{245}+2D_{146},$
and $D_{246}$ (slightly strange choices of $D'_{136}$ and $D_{245}'$
are for the group action on the hypersurface $\{G=0\}$ defined just
below. See Proposition \ref{prop:GL}). Then we consider the hypersurface
in the affine space with coordinates $p_{1},p_{2},p_{3},p_{4}$,$u$,$v$,
$t_{1}$, $t_{2}$, and $L_{123},L_{124},L_{125},L_{126},L_{135},L_{136},L_{245},L_{246}$
defined by the following polynomial:{\small{}
\begin{align}
G & :=L_{123}D_{123}+L_{124}D_{124}+L_{125}D_{125}+L_{126}D_{126}\nonumber \\
 & +L_{135}D_{135}+3L_{136}D_{136}'+3L_{245}D_{245}'+L_{246}D_{246}.\label{eq:G}
\end{align}
}{\small\par}

Inspired by the theory of unprojection, we define the following affine
variety with nine equations.
\begin{defn}
\label{def:Pi}(1) We consider the affine space $\mathbb{A}^{19}$
with coordinates $p_{1},p_{2},p_{3},p_{4},$$u$,$v$,$t_{1}$,$t_{2}$,
$L_{123},L_{124},L_{125},L_{126},L_{135},L_{136},L_{245},L_{246}$
and $s_{1},s_{2},s_{3}$. In this affine space, we define an affine
scheme $\Pi_{\mathbb{A}}^{15}$ by the following nine polynomials
$F_{1},\dots,F_{9}$:\newpage{\small{}
\begin{align*}
F_{1} & =(us_{1}-p_{1}s_{2}-p_{2}s_{3})-L_{126}(-p_{2}p_{3}+p_{1}p_{4})-L_{136}(p_{1}^{2}+t_{1}p_{2}^{2}+p_{2}u)\\
 & -L_{245}(2p_{1}p_{3}+2t_{1}p_{2}p_{4}+p_{4}u+p_{2}v)-L_{246}(p_{3}^{2}+t_{1}p_{4}^{2}+p_{4}v),\\
\\
F_{2} & =(vs_{1}-p_{3}s_{2}-p_{4}s_{3})+L_{125}(-p_{2}p_{3}+p_{1}p_{4})+L_{135}(p_{1}^{2}+t_{1}p_{2}^{2}+p_{2}u)\\
 & +L_{136}(2p_{1}p_{3}+2t_{1}p_{2}p_{4}+p_{4}u+p_{2}v)+L{}_{245}(p_{3}^{2}+t_{1}p_{4}^{2}+p_{4}v),\\
\\
F_{3} & =\big(-t_{2}p_{2}s_{1}+(u+t_{1}p_{2})s_{2}-p_{1}s_{3}\big)-L_{124}(-p_{2}p_{3}+p_{1}p_{4})-L_{136}(t_{2}p_{2}^{2}+p_{1}u)\\
 & -L_{245}(2t_{2}p_{2}p_{4}+p_{3}u+p_{1}v)-L_{246}(t_{2}p_{4}^{2}+p_{3}v),\\
\\
F_{4} & =\big(-t_{2}p_{4}s_{1}+(v+t_{1}p_{4})s_{2}-p_{3}s_{3}\big)+L_{123}(-p_{2}p_{3}+p_{1}p_{4})+L_{135}(t_{2}p_{2}^{2}+p_{1}u)\\
 & +L_{136}(2t_{2}p_{2}p_{4}+p_{3}u+p_{1}v)+L_{245}(t_{2}p_{4}^{2}+p_{3}v),\\
\\
F_{5} & =\big(-t_{2}p_{1}s_{1}+(t_{1}p_{1}-t_{2}p_{2})s_{2}+(u+t_{1}p_{2})s_{3}\big)\\
 & +L_{123}(p_{1}^{2}+t_{1}p_{2}^{2}+p_{2}u)+L_{124}(p_{1}p_{3}+t_{1}p_{2}p_{4}+p_{2}v)\\
 & -L_{125}(t_{2}p_{2}^{2}+p_{1}u)-L_{126}(t_{2}p_{2}p_{4}+p_{1}v)\\
 & -L_{136}(t_{1}p_{2}u-t_{2}p_{1}p_{2}+u^{2})-L_{245}\big(t_{1}(p_{4}u+p_{2}v)-t_{2}(p_{2}p_{3}+p_{1}p_{4})+2uv\big)\\
 & -L_{246}(t_{1}p_{4}v-t_{2}p_{3}p_{4}+v^{2}),\\
\\
F_{6} & =(-t_{2}p_{3}s_{1}+(t_{1}p_{3}-t_{2}p_{4})s_{2}+(v+t_{1}p_{4})s_{3})\\
 & +L_{123}(t_{1}p_{2}p_{4}+p_{1}p_{3}+p_{4}u)+L_{124}(p_{3}^{2}+t_{1}p_{4}^{2}+p_{4}v)\\
 & -L_{125}(t_{2}p_{2}p_{4}+p_{3}u)-L_{126}(t_{2}p{}_{4}^{2}+p_{3}v)\\
 & +L_{135}(t_{1}p_{2}u-t_{2}p_{1}p_{2}+u^{2})+L_{136}\big(t_{1}(p_{4}u+p_{2}v)-t_{2}(p_{2}p_{3}+p_{1}p_{4})+2uv\big)\\
 & +L_{245}(t_{1}p_{4}v-t_{2}p_{3}p_{4}+v^{2}),\\
\end{align*}
\begin{align*}
 & F_{7}=(-s_{2}^{2}+s_{1}s_{3})+\big((L_{123}p_{2}+L_{124}p_{4})s_{1}-(L_{125}p_{2}+L_{126}p_{4})s_{2}\big)\\
 & +L_{123}(L_{136}p_{2}^{2}+2L_{245}p_{2}p_{4}+L_{246}p_{4}^{2})-L_{124}(L_{135}p_{2}^{2}+2L_{136}p_{2}p_{4}+L_{245}p_{4}^{2})\\
 & -L_{125}\big(L_{136}p_{1}p_{2}+L_{245}(p_{2}p_{3}+p_{1}p_{4})+L_{246}p_{3}p_{4}\big)+L_{126}\big(L_{135}p_{1}p_{2}+L_{136}(p_{2}p_{3}+p_{1}p_{4})+L_{245}p_{3}p_{4}\big)\\
 & +\big(L_{135}L{}_{245}-L_{136}^{2}\big)(-p_{1}^{2}+2up_{2}+t_{1}p_{2}^{2})+(L_{135}L_{246}-L{}_{136}L_{245})(-p_{1}p_{3}+up_{4}+vp_{2}+t_{1}p_{2}p_{4})\\
 & +\big(L{}_{136}L_{246}-L_{245}^{2}\big)(-p_{3}^{2}+2vp_{4}+t_{1}p_{4}^{2}),\\
\\
 & F_{8}=(s_{2}s_{3}+t_{1}s_{1}s_{2}-t_{2}s_{1}^{2})+s_{1}(L_{123}p_{1}+L_{124}p_{3})-s_{2}(L_{125}p_{1}+L_{126}p_{3})\\
 & +L_{123}\big(L{}_{136}p_{1}p_{2}+L{}_{245}(p_{1}p_{4}+p_{3}p_{2})+L_{246}p_{3}p_{4}\big)-L_{124}\big(L_{135}p_{1}p_{2}+L{}_{136}(p_{1}p_{4}+p_{3}p_{2})+L{}_{245}p_{3}p_{4}\big)\\
 & -L_{125}(L{}_{136}p_{1}^{2}+2L_{245}p_{1}p_{3}+L_{246}p_{3}^{2})+L_{126}(L_{135}p_{1}^{2}+2L{}_{136}p_{1}p_{3}+L{}_{245}p_{3}^{2})\\
 & +\big(-L_{136}^{2}+L_{135}L{}_{245}\big)(2up_{1}+2t_{1}p_{1}p_{2}-t_{2}p_{2}^{2})\\
 & +(L_{135}L_{246}-L{}_{136}L{}_{245})\big((vp_{1}+up_{3})+t_{1}(p_{1}p_{4}+p_{3}p_{2})-t_{2}p_{2}p_{4}\big)\\
 & +(L_{246}L{}_{136}-L_{245}^{2})(2vp_{3}+2t_{1}p_{3}p_{4}-t_{2}p_{4}^{2}),\\
\\
 & F_{9}=(t_{1}s_{2}^{2}+s_{3}^{2}-s_{1}s_{2}t_{2})+s_{1}(L_{123}u+L_{124}v)-s_{2}(L_{125}u+L_{126}v)\\
 & +(L_{124}L_{125}-L_{123}L_{126})(-p_{2}p_{3}+p_{1}p_{4})\\
 & -L_{123}\big(L_{246}(p_{3}^{2}+t_{1}p_{4}^{2})+L{}_{136}(p_{1}^{2}+t_{1}p_{2}^{2})+2L{}_{245}(p_{1}p_{3}+t_{1}p_{2}p_{4})\big)\\
 & +L_{124}\big(L_{135}(p_{1}^{2}+t_{1}p_{2}^{2})+2L{}_{136}(p_{1}p_{3}+t_{1}p_{2}p_{4})+L{}_{245}(p_{3}^{2}+t_{1}p_{4}^{2})\big)\\
 & +L_{125}t_{2}(L_{246}p_{4}^{2}+L{}_{136}p_{2}^{2}+2L{}_{245}p_{2}p_{4})-L_{126}t_{2}(L_{135}p_{2}^{2}+2L{}_{136}p_{2}p_{4}+L{}_{245}p_{4}^{2})\\
 & +\big(-L_{136}^{2}+L_{135}L{}_{245}\big)(u^{2}+2t_{2}p_{1}p_{2})+(L_{135}L_{246}-L{}_{136}L{}_{245})\big(uv+t_{2}(p_{2}p_{3}+p_{1}p_{4})\big)\\
 & +\big(L_{246}L{}_{136}-L_{245}^{2}\big)(v^{2}+2t_{2}p_{3}p_{4}).\\
\end{align*}
}{\small\par}

\vspace{3pt}

\noindent (2) We denote by $\Pi_{\mathbb{A}}^{14}$ the affine scheme
obtained from $\Pi_{\mathbb{A}}^{15}$ by setting $L_{246}=1$. 
\end{defn}

\vspace{3pt}

Following \cite{key-2} and \cite{key-3}, we show in the section
\ref{sec:Unprojection} that $\Pi_{\mathbb{A}}^{15}$ is a $15$-dimensional
normal Gorenstein varieties of codimension $4$ in $\mathbb{A}^{19}$.
Moreover, we show that $\Pi_{\mathbb{A}}^{15}$ has only terminal
singularities (Proposition \ref{prop:MoreSing}). 

\subsection{Main Theorem}

The main result of this paper is the following theorem, for which
we use for the classes of $\mathbb{Q}$-Fano threefolds the assigned
numbers in the database \cite{key-6}. We prepare one convention;
by a \textit{numerical data} for a prime $\mathbb{Q}$-Fano threefold
$X$ of No. $*$, we mean the quadruplet consisting of the Fano index
of $X$, the genus $g(X):=h^{0}(-K_{X})-2$, the basket of singularities
of $X$, and the Hilbert numerator of $X$ with respect to $-K_{X}$,
which are given in the database \cite{key-6}. 
\begin{thm}
\label{thm:main}The following assertions hold:\vspace{3pt}

\noindent$(1)\,(1\text{-}{\rm 1)}$ There exists a quasi-smooth prime
$\mathbb{Q}$-Fano threefold $X$ with ${\rm ac}_{X}=4$ for the eight
numerical data No.~$308$,~$501$,~$512$,~$550$,~$577$,~$872$,~$878$,
and $1766$ of \cite{key-6} such that it is obtained as a weighted
complete intersection from the weighted projectivization $\Pi_{\mathbb{P}}^{13}$
of $\Pi_{\mathbb{A}}^{14}$ whose weights of coordinates given below,
or the cone $\Pi_{\mathbb{P}}^{14}$over $\Pi_{\mathbb{P}}^{13}$
with a weight one coordinate added. 

\vspace{3pt}

We denote by $w(*)$ the weight of the coordinate $*$. 

\vspace{3pt}

\noindent\textbf{\textup{Weights for No.\,308,~501,~512,~550,~872:}}

{\small{}
\begin{align*}
 & \left(\begin{array}{cc}
w(u) & w(v)\\
w(p_{1}) & w(p_{3})\\
w(p_{2}) & w(p_{4})
\end{array}\right)=\left(\begin{array}{cc}
d-1 & d+1\\
d-2 & d\\
d-3 & d-1
\end{array}\right),\\
 & \left(\begin{array}{ccc}
w(s_{1}) & w(s_{2}) & w(s_{3})\end{array}\right)=\left(\begin{array}{ccc}
d+1 & d+2 & d+3\end{array}\right),\,w(t_{1})=2,\,w(t_{2})=3,\\
 & \left(\begin{array}{cc}
w(L_{123}) & w(L_{125})\\
w(L_{124}) & w(L_{126})
\end{array}\right)=\left(\begin{array}{cc}
6 & 5\\
4 & 3
\end{array}\right),\,w(L_{135})=6,\,w(L{}_{136})=4,\,w(L_{245})=2,\\
 & \Pi_{\mathbb{P}}^{13}\subset\mathbb{P}^{17}(2^{2},3^{2},4^{2},5,6^{2},d-3,d-2,d-1^{2},d,d+1^{2},d+2,d+3),
\end{align*}
}where $d=7,7,6,5,4$ for No.~$308$,~$501$,~$512$,~$550$,~$872$
respectively.

\vspace{3pt}

\noindent\textbf{\textup{Weights for No.\,577,~878,~1766:}}

{\small{}
\begin{align*}
 & \left(\begin{array}{cc}
w(u) & w(v)\\
w(p_{1}) & w(p_{3})\\
w(p_{2}) & w(p_{4})
\end{array}\right)=\left(\begin{array}{cc}
d+1 & d+2\\
d & d+1\\
d-1 & d
\end{array}\right),\\
 & \left(\begin{array}{ccc}
w(s_{1}) & w(s_{2}) & w(s_{3})\end{array}\right)=\left(\begin{array}{ccc}
d+1 & d+2 & d+3\end{array}\right),\,w(t_{1})=2,\,w(t_{2})=3,\\
 & \left(\begin{array}{cc}
w(L_{123}) & w(L_{125})\\
w(L_{124}) & w(L_{126})
\end{array}\right)=\left(\begin{array}{cc}
4 & 3\\
3 & 2
\end{array}\right),\,w(L_{135})=3,\,w(L{}_{136})=2,\,w(L_{245})=1,\\
 & \Pi_{\mathbb{P}}^{13}\subset\mathbb{P}^{17}(1,2^{3},3^{4},4,d-1,d^{2},d+1^{3},d+2^{2},d+3),
\end{align*}
}where $d=4,3,2$ for No.~$577$,~$878$,~$1766$ respectively.
\vspace{3pt}

Moreover, descriptions of $X$ in each classes are presented as in
Table \ref{Table1}.

\vspace{3pt}

\noindent$(1\text{-}2)$ For such an $X$ of No.~$501$, $512$,
the locus $\mathbb{P}(s_{1},s_{2},s_{3})\cap X$ set-theoretically
consists of the $s_{1}$-point, which is the unique highest index
cyclic quotient singularity of $X$. A projection of $X$ is induced
by the projection of $\Pi_{\mathbb{P}}^{14}$ from the intersection
with the locus $\mathbb{P}(s_{1},s_{2},s_{3})$ to the cone over the
hypersurface $\{G=0\}|_{L_{246}=1}$ with a weight one coordinate
added. Moreover this projection of $X$ is birational. 

\vspace{3pt}

\noindent$(1\text{-}3)$ For such an $X$ of No.~$550$,~$872$,~$577$,~$878$~or~$1766$,
there exists a projection of $X$ to a hypersurface of weight $18$,~$15$,~$18$,~$15$
or $12$ respectively with the center being any one of the highest
index cyclic quotient singularities of $X$. Moreover this projection
is birational (for more detailed descriptions of the projection as
in $(1\text{{-}}2)$, see the subsections $\ref{subsec:No.-550,-872}$
and $\ref{subsec:No.-577,-878,}$). 

\vspace{3pt}

\noindent $(2)$ A general member $T$ of $|-K_{X}|$ for a general
$X$ as in $(1)$ is a quasi-smooth $K3$ surface with only Du Val
singularities of type $A$. Moreover, $T$ has only singularities
at points where $X$ is singular and if $X$ has a $1/\alpha\,(\beta,-\beta$,1)-singularity
at a point for some $\alpha,\,\beta\in\mathbb{N}$ with $(\alpha,\beta)=1$,
then has a $1/\alpha\,(\beta,-\beta)$-singularity there.
\end{thm}

\begin{table}
\caption{Descriptions of $\mathbb{Q}$-Fano threefolds $X$}
\label{Table1}

{\scriptsize{

\renewcommand{\arraystretch}{2.5}

\begin{tabular}{|c|c|c|c|}
\hline 
No. & $\mathbb{P}_{X}$ & Baskets of singularities & $X\subset\Pi$\tabularnewline
\hline 
\hline 
308 & $\mathbb{P}(1,5,6^{2},7,8,9,10)$ & $1/2(1,1,1),1/3(1,1,2),1/5(1,2,3),$ & $\Pi_{\mathbb{P}}^{14}\cap(2)^{2}\cap(3)^{2}\cap(4)^{3}\cap(5)\cap(6)^{2}\cap(8)$\tabularnewline
 &  & $2\times1/6(1,1,5)$ & \tabularnewline
\hline 
501 & $\mathbb{P}(1,3,6,7,8^{2},9,10)$ & $1/2(1,1,1),4\times1/3(1,1,2),1/8(1,1,7)$ & $\Pi_{\mathbb{P}}^{14}\cap(2)^{2}\cap(3)\cap(4)^{3}\cap(5)^{2}\cap(6)^{3}$\tabularnewline
\hline 
512 & $\mathbb{P}(1,3,5,6,7^{2},8,9)$ & $3\times1/3(1,1,2),1/5(1,2,3),1/7(1,1,6)$ & $\Pi_{\mathbb{P}}^{14}\cap(2)^{2}\cap(3)^{2}\cap(4)^{3}\cap(5)^{2}\cap(6)^{2}$\tabularnewline
\hline 
550 & $\mathbb{P}(1,3,4,5,6^{2},7,8)$ & \multirow{1}{*}{$1/2(1,1,1),3\times1/3(1,1,2),1/4(1,1,3),$} & $\Pi_{\mathbb{P}}^{14}\cap(2)^{3}\cap(3)^{2}\cap(4)^{3}\cap(5)\cap(6)^{2}$\tabularnewline
 &  & \multirow{1}{*}{$1/6(1,1,5)$} & \tabularnewline
\hline 
872 & $\mathbb{P}(1,3^{2},4,5^{2},6,7)$ & $5\times1/3(1,1,2),1/5(1,1,4)$ & $\Pi_{\mathbb{P}}^{13}\cap(2)^{3}\cap(3)^{2}\cap(4)^{2}\cap(5)\cap(6)^{2}$\tabularnewline
\hline 
577 & $\mathbb{P}(1,3,4,5^{2},6^{2},7)$ & $1/2(1,1,1),3\times1/3(1,1,2),2\times1/5(1,1,4)$ & $\Pi_{\mathbb{P}}^{13}\cap(2)^{3}\cap(3)^{4}\cap(4)^{2}\cap(5)$\tabularnewline
\hline 
878 & $\mathbb{P}(1,3^{2},4^{2},5^{2},6)$ & $4\times1/3(1,1,2),2\times1/4(1,1,3)$ & $\Pi_{\mathbb{P}}^{13}\cap(2)^{4}\cap(3)^{4}\cap(4)^{2}$\tabularnewline
\hline 
1766 & $\mathbb{P}(1,2,3^{3},4^{2},5)$ & $2\times1/2(1,1,1),5\times1/3(1,1,2)$ & $\Pi_{\mathbb{P}}^{13}\cap(1)\cap(2)^{4}\cap(3)^{4}\cap(4)$\tabularnewline
\hline 
\end{tabular}

}}\vspace{1cm}
\end{table}

The existence of $\mathbb{Q}$-Fano threefolds as in Theorem \ref{thm:main}
(1) is suggested in \cite[the subsec.5.2]{key-3} except No.308. It
is a bonus of our key variety construction to show the existence of
a $\mathbb{Q}$-Fano threefold $X$ of No.308. Note that such an $X$
was shown to be birationally superrigid by \cite{key-7}. 

Finally in this subsection, we mention our future plan related to
the present and previous works. In this paper, we introduce another
affine variety $H_{\mathbb{A}}^{13}$. We treat it as a supporting
character to investigate a projective geometric property of $\Pi_{\mathbb{A}}^{15}$.
In the forthcoming paper \cite{key-16}, we show that $H_{\mathbb{A}}^{13}$
has a closed subvariety isomorphic to the cluster variety of type
$C_{2}$, where we refer to \cite{key-11} for the definition of the
cluster variety of type $C_{2}$. Then we show that several weighted
projectivizations of closed subvarieties of $H_{\mathbb{A}}^{13}$
produce more general prime $\mathbb{Q}$-Fano threefolds than those
obtained from the cluster variety of type $C_{2}$. Moreover, a weighted
projectivization of $H_{\mathbb{A}}^{13}$ itself produces a prime
$\mathbb{Q}$-Fano threefold of No.20652, which was not obtained from
the cluster variety of type $C_{2}$. 

We also plan to construct prime $\mathbb{Q}$-Fano threefolds $X$
with ${\rm ac}_{X}=4$ using key varieties related to $\mathbb{P}^{1}\times\mathbb{P}^{1}\times\mathbb{P}^{1}$-fibrations.

\subsection{Structure of the paper}

The aim of this paper is twofold; the construction of $\mathbb{Q}$-Fano
threefolds as in Theorem \ref{thm:main}, and the description of the
key variety $\Pi_{\mathbb{A}}^{15}$. 

After obtaining preliminary results in the section \ref{sec:Preliminary-results},
we show in the section \ref{sec:Unprojection} that the affine coordinate
ring of $\Pi_{\mathbb{A}}^{15}$ is a domain (Proposition \ref{prop:Unproj})
and is Gorenstein (Proposition \ref{prop:Gor}) using the theory of
unprojection and following the papers \cite{key-2} and \cite{key-3}.
In the sections \ref{sec:--and--actions} and \ref{sec:Singular-locus-Pi},
we describe a group action on $\Pi_{\mathbb{A}}^{15}$ and determine
the singular locus of $\Pi_{\mathbb{A}}^{15}$, respectively. In the
section \ref{sec:Factoriality-of-the}, we show that the affine coordinate
ring of $\Pi_{\mathbb{A}}^{15}$ is factorial (Proposition \ref{prop:UFD}),
and as an application of this, we describe the divisor class groups
of weighted projectivizations of $\Pi_{\mathbb{A}}^{15}$ and weighted
compelete intersections in them (Proposition \ref{prop:Pic1}). In
the section \ref{sec:Graded--resolution}, we calculate the graded
free resolution of the affine coordinate ring of $\Pi_{\mathbb{A}}^{15}$
(Proposition \ref{prop:916}) and observe that it is compatible with
the data given in \cite{key-6} (Corollaries \ref{cor:QFanoData}
and \ref{cor:acX4}). In the section \ref{sec:Proof-of-Theorem},
we show Theorem \ref{thm:main}. To investigate properties of $\Pi_{\mathbb{A}}^{15}$
more, we introduce and describe another affine variety $H_{\mathbb{A}}^{13}$
in the section \ref{sec:Affine-variety}. In the section \ref{sec:More-on-geometry},
we describe a $\mathbb{P}^{2}\times\mathbb{P}^{2}$-fibration structure
of a partial projectivization of $\Pi_{\mathbb{A}}^{15}$ (Proposition
\ref{prop:P2P2Pi}) and detailed singularities of $\Pi_{\mathbb{A}}^{15}$
(Proposition \ref{prop:MoreSing}).

\vspace{3pt}

\noindent\textbf{Acknowledgment: }The author owes many inspiring
calculations in the paper to Professor Shinobu Hosono. The author
wishes to thank him for his generous cooperations from the time of
writing the previous paper. This work is supported in part by Grant-in
Aid for Scientific Research (C) 16K05090. The author would like to
dedicate this paper to his father, who passed away while the author
was preparing the paper.

\section{Preliminary results\label{sec:Preliminary-results}}

We set 
\begin{align*}
S_{G} & =\mathbb{C}[p_{1},p_{2},p_{3},p_{4},u,v,t_{1},t_{2},L_{123},L_{124},L_{125},L_{126},L_{135},L_{136},L_{245},L_{246}],
\end{align*}
and 
\[
\tilde{S}_{G}=\mathbb{C}[p_{1},p_{2},p_{3},p_{4},t_{1},w,L_{123},L_{124},L_{125},L_{126},L_{135},L_{136},L_{245},L_{246}].
\]
Since these are related with the polynomial $G$ as in (\ref{eq:G}),
we put the subscripts $G$ for them. Let $h\colon S_{G}\to\tilde{S}_{G}$
be the ring homomorphism defined with the equalities in (\ref{eq:zb})
and the other coordinates being unchanged. Let $I_{D}:={\rm ker\,}h$
and $D$ the subscheme of $\mathbb{A}^{16}={\rm Spec}\,S_{G}$ defined
by $I_{D}$. The subscheme $D$ is irreducible and reduced since $S_{G}/I_{D}$
is a subring of the polynomial ring $\tilde{S}_{G}$.

In this section, we study several properties of $D$ and $\{G=0\}$.
All the results in this section follow from direct computations, hence
their proofs are omitted.
\begin{prop}
\label{prop:EqD} The ideal $I_{D}$ is generated by all the $3\times3$
minors of the $3\times6$ matrix $\mathsf{M}$ as in $(\ref{eq:M})$.
Consequently, by the observation in the subsection $\ref{subsec:Affine-key-varieties}$,
$I_{D}$ is actually generated by $D_{123}$,~$D_{124}$,~$D_{125}$,~$D_{126}$,~$D_{135}$,~$D_{136}'$,~$D_{245}'$,~$D_{246}$.
\end{prop}

We set $\tilde{D}={\rm Spec\,}\tilde{S}_{G}\simeq\mathbb{A}^{14}$.
\begin{prop}
\label{prop:SingD} Let $\nu_{D}\colon\tilde{D}\to D$ be the induced
morphism. Then $\nu_{D}$ is the normalization (and hence is surjective
and $D$ is irreducible). The singular locus of $D$ is the common
zero of the $2\times2$ minors of the matrix $\mathsf{M}$. In particular,
codimension of ${\rm Sing}\,D$ is two in $D$.
\end{prop}

We set $R_{D}:=S_{G}/I_{D}$. By the definition of $h\colon S_{G}\to\tilde{S}_{G}$
(cf.~(\ref{eq:zb})), we see that $\tilde{S}_{G}$ is generated by
$1,w,w^{2}$ as an $R_{D}$-module. Therefore an $R_{D}$-module surjective
homomorphism $R_{D}^{\oplus3}\to\tilde{S}_{G}$ is induced. We may
check that the composite of this $R_{D}^{\oplus3}\to\tilde{S}_{G}$
with the map $R_{D}^{\oplus6}\to R_{D}^{\oplus3}$ defined by multiplying
the matrix $M$ is a zero map. Therefore we obtain a complex $R_{D}^{\oplus6}\stackrel{\mathsf{M}\times}{\to}R_{D}^{\oplus3}\to\tilde{S}_{G}\to0$.
\begin{prop}
\label{prop:The-first-syzygy}The above obtained complex $R_{D}^{\oplus6}\stackrel{\mathsf{M}\times}{\to}R_{D}^{\oplus3}\to\tilde{S}_{G}\to0$
is exact.
\end{prop}

\begin{prop}
\label{prop:irredG}The polynomial $G$ is irreducible and the affine
hypersurface $\{G=0\}$ is normal. 
\end{prop}

\section{Gorensteinness of the affine coordinate ring of $\Pi_{\mathbb{A}}^{15}$\label{sec:Unprojection}}

Let $S_{\Pi}:=S_{G}[s_{1},s_{2},s_{3}]$, which is a polynomial ring
over $\mathbb{C}$ with $19$ variables. Let $I_{\Pi}$ be the ideal
of $S_{\Pi}$ generated by the nine polynomials $F_{1},\dots,F_{9}$
as in Definition \ref{def:Pi} and set $R_{\Pi}:=S_{\Pi}/I_{\Pi}$,
which is the affine coordinate ring of $\Pi_{\mathbb{A}}^{15}$. 

Let $R_{G}=S_{G}/(G)$ and $Q_{G}$ the quotient field of $R_{G}$
(cf.~Proposition \ref{prop:irredG}). Note that a ring homomorphism
$h'\colon R_{G}\to\tilde{S_{G}}$ is induced from $h\colon S_{G}\to\tilde{S_{G}}$
since $G\in I_{D}$ by (\ref{eq:G}) and Proposition \ref{prop:EqD}.

\subsection{Unprojection}

Following closely \cite{key-3} (the origin of the arguments there
is in \cite{key-2}), we show the following:
\begin{prop}
\label{prop:Unproj}Let $J_{D}:=I_{D}/(G)$. The affine coordinate
ring $R_{\Pi}$ is isomorphic to the unprojection ring $R_{G}[J_{D}^{-1}]$,
where $J_{D}^{-1}:=\{f\in Q_{G}\mid fJ_{D}\in R_{G}\}$ and $R_{G}[J_{D}^{-1}]$
is a subring of $Q_{G}$ generated by $J_{D}^{-1}$ over $R_{G}$.
In particular, $R_{\Pi}$ is a domain.
\end{prop}

\begin{proof}
We closely follow the arguments in \cite[the subsec.~2.3]{key-3}.
We only give a sketch of the proof omitting the arguments there which
work in our situation without change. 

We prepare some notation. Note that the equations $F_{1}=0$,~$F_{2}=0$,~$F_{5}=0$,~$F_{6}=0$
of $\Pi_{\mathbb{A}}^{15}$ is equivalent to $F_{1}=0$,~$F_{2}=0$,~$F_{5}+t_{1}F_{1}=0$,~$F_{6}+t_{1}F_{2}=0$.
We may write $F_{1}=0,$~$F_{2}=0$,~$F_{3}=0$,~$F_{4}=0$,~$F_{5}+t_{1}F_{1}=0$,~$F_{6}+t_{1}F_{2}=0$
in the following form: 
\begin{equation}
\left(\begin{array}{ccc}
s_{1} & s_{2} & s_{3}\end{array}\right)\mathsf{M}=\left(\begin{array}{cccccc}
H_{1} & H_{2} & H_{3} & H_{4} & H_{5} & H_{6}\end{array}\right).\label{eq:s1s2s3M}
\end{equation}
Let $\mathsf{M}_{ijk}^{\dagger}$ be the adjoint matrix of $\mathsf{M}_{ijk}$.

As for the arguments corresponding to those in the proof of \cite[Lem.~2.3.4 and 2.3.6]{key-3},
we take a short cut referring to the explicit form of the equations
$F_{i}=0\,(1\leq i\leq9)$ from the beginning. Taking account of $F_{1}$,~$F_{2}$,~$F_{3}$,
we define three rational functions $s_{1},s_{2},s_{3}\in Q_{G}$ solving
the matrix equations $\left(\begin{array}{ccc}
s_{1} & s_{2} & s_{3}\end{array}\right)\mathsf{M}_{123}=\left(\begin{array}{ccc}
H_{1} & H_{2} & H_{3}\end{array}\right)$, where we consider entries of $\mathsf{M}$ and $H_{1}$,~$H_{2}$,~$H_{3}$
are elements of $R_{G}$. In other words, we define $\left(\begin{array}{ccc}
s_{1} & s_{2} & s_{3}\end{array}\right)=\frac{1}{D_{123}}\left(\begin{array}{ccc}
H_{1} & H_{2} & H_{3}\end{array}\right)\mathsf{M}_{123}^{\dagger},$ where we note that this is well-defined since $D_{123}\not\equiv0$
on $\{G=0\}$ by (\ref{eq:G}) and Proposition \ref{prop:irredG}.
We will show that $s_{1},s_{2},s_{3}$ are elements of $J_{D}^{-1}$.
By straightforward calculations, we can verify the equality 
\[
D_{ijk}\left(\begin{array}{ccc}
H_{p} & H_{q} & H_{r}\end{array}\right)\mathsf{M}_{pqr}^{\dagger}=D_{pqr}\left(\begin{array}{ccc}
H_{i} & H_{j} & H_{k}\end{array}\right)\mathsf{M}_{ijk}^{\dagger}
\]
 holds in $R_{G}$ for any $1\leq i<j<k\leq6$ and $1\leq p<q<r\leq6$.
By the definition of $s_{1},s_{2},s_{3}$, we have $D_{123}\left(\begin{array}{ccc}
s_{1} & s_{2} & s_{3}\end{array}\right)=\left(\begin{array}{ccc}
H_{1} & H_{2} & H_{3}\end{array}\right)\mathsf{M}_{123}^{\dagger}$. Therefore we have 
\[
D_{ijk}D_{123}\left(\begin{array}{ccc}
s_{1} & s_{2} & s_{3}\end{array}\right)=D_{123}\left(\begin{array}{ccc}
H_{i} & H_{j} & H_{k}\end{array}\right)\mathsf{M}_{ijk}^{\dagger}
\]
and then obtain 
\begin{equation}
\left(\begin{array}{ccc}
s_{1} & s_{2} & s_{3}\end{array}\right)=\frac{1}{D_{ijk}}\left(\begin{array}{ccc}
H_{i} & H_{j} & H_{k}\end{array}\right)\mathsf{M}_{ijk}^{\dagger}\label{eq:s1s2s3}
\end{equation}
for any $1\leq i<j<k\leq6$. Hence $s_{1},s_{2},s_{3}$ are actually
elements of $J_{D}^{-1}$ by Proposition \ref{prop:EqD}.

From the equality (\ref{eq:s1s2s3}), we have $\left(\begin{array}{ccc}
s_{1} & s_{2} & s_{3}\end{array}\right)\mathsf{M}_{ijk}=\left(\begin{array}{ccc}
H_{i} & H_{j} & H_{k}\end{array}\right)$ for any $1\leq i<j<k\leq6$. Thus we have shown that the equalities
$F_{i}=0\,(1\leq i\leq6)$ hold in $R_{G}[J_{D}^{-1}]$ with the definition
of $s_{1},$~$s_{2}$,~$s_{3}$ as above. Moreover we may verify
$F_{7}=F_{8}=F_{9}=0$ also hold in $R_{G}[J_{D}^{-1}]$ substituting
$s_{1},s_{2},s_{3}$ in $F_{7}$,~$F_{8}$,~$F_{9}$ with the r.h.s.
of (\ref{eq:s1s2s3}) with $(i,j,k)=(1,2,3)$. Thus we have linear
equations $F_{1}=\cdots=F_{6}=0$ and quadratic equations $F_{7}=F_{8}=F_{9}=0$
with respect to $s_{1},s_{2},s_{3}$ in $R_{G}[J_{D}^{-1}]$ corresponding
to \cite[Lem.~2.3.4 and 2.3.6]{key-3}. 

We show that $R_{G}[J_{D}^{-1}]$ is generated by $s_{1}$,~$s_{2}$,~$s_{3}$
as $R_{G}$-algebra. Similarly to \cite[Prop.~2.3.1 and Lem.~2.3.4]{key-3},
let $s_{1}'$,~$s_{2}'$,~$s_{3}'\in J_{D}^{-1}$ be lifts of $1$,~$w$,~$w^{2}\in\tilde{S_{G}}$
by the residue map $r\colon J_{D}^{-1}\to\tilde{S}_{G}$ defined as
in \cite[Lem.~2.3.3]{key-3}. Then $R_{G}[J_{D}^{-1}]$ is generated
by $s_{1}'$,~$s_{2}'$,~$s_{3}'$ as $R_{G}$-algebra and there
exists a relation 
\begin{equation}
\left(\begin{array}{ccc}
s_{1}' & s_{2}' & s_{3}'\end{array}\right)\mathsf{M}=\left(\begin{array}{cccccc}
H_{1}' & H_{2}' & H_{3}' & H_{4}' & H_{5}' & H_{6}'\end{array}\right)\label{eq:s1's2's3'M}
\end{equation}
 for some $H_{i}'\in R_{G}\,(1\leq i\leq6)$ as in \cite[Lem.~2.3.4]{key-3}.
Note that, similarly to (\ref{eq:s1s2s3}), we obtain from (\ref{eq:s1's2's3'M})

\begin{equation}
\left(\begin{array}{ccc}
s_{1}' & s_{2}' & s_{3}'\end{array}\right)=\frac{1}{D_{ijk}}\left(\begin{array}{ccc}
H_{i}' & H_{j}' & H_{k}'\end{array}\right)\mathsf{M}_{ijk}^{\dagger}\label{eq:s1's2's3'}
\end{equation}
for any $1\leq i<j<k\leq6$. Since $r(s_{1}'),r(s_{2}'),r(s_{3}')$
generates $\tilde{S}_{G}$ as an $R_{G}$-module, there exists a $3\times3$
matrix $\mathsf{A}$ whose entries are elements of $R_{G}$ such that
\begin{equation}
\left(\begin{array}{ccc}
r(s_{1}) & r(s_{2}) & r(s_{3})\end{array}\right)=\left(\begin{array}{ccc}
r(s_{1}') & r(s_{2}') & r(s_{3}')\end{array}\right)h'(\mathsf{A}),\label{eq:ss'}
\end{equation}
where we denote by $h'(\mathsf{A})$ the matrix whose $(i,j)$ entries
are the $h'$-images of $(i,j)$ entries of $\mathsf{A}$~($1\leq i,j\leq3)$.
Therefore all the entries of $\left(\begin{array}{ccc}
s_{1} & s_{2} & s_{3}\end{array}\right)-\left(\begin{array}{ccc}
s_{1}' & s_{2}' & s_{3}'\end{array}\right)\mathsf{A}$ are elements of $R_{G}.$ By (\ref{eq:s1s2s3}) and (\ref{eq:s1's2's3'})
with $(i,j,k)=(1,2,3)$, this implies that $D_{123}$ divides in $R_{G}$
all the entries of the vector 
\[
\left(\begin{array}{ccc}
H_{1} & H_{2} & H_{3}\end{array}\right)\mathsf{M}_{123}^{\dagger}-\left(\begin{array}{ccc}
H_{1}' & H_{2}' & H_{3}'\end{array}\right)\mathsf{M}_{123}^{\dagger}\mathsf{A.}
\]
Thus we have 

\begin{align}
 & \left(\begin{array}{ccc}
h'(H_{1}) & h'(H_{2}) & h'(H_{3})\end{array}\right)h'(\mathsf{M}_{123}^{\dagger})=\label{eq:M123}\\
 & \quad\left(\begin{array}{ccc}
h'(H_{1}') & h'(H_{2}') & h'(H_{3}')\end{array}\right)h'(\mathsf{M}_{123}^{\dagger})h'(\mathsf{A})\nonumber 
\end{align}
since $h'(D_{123})=0$, where $h'(\mathsf{M}_{123}^{\dagger})$ is
defined for $\mathsf{M}_{123}^{\dagger}$ similarly to $h'(\mathsf{A)}$.
We set $h'(\mathsf{M}_{123}^{\dagger})=\left(\begin{array}{ccc}
\bm{m}_{1} & \bm{m}_{2} & \bm{m}_{3}\end{array}\right)$. By an explicit calculation using (\ref{eq:zb}), we see that $\bm{m}_{2}=w\bm{m}_{1}$
and $\bm{m}_{3}=w^{2}\bm{m}_{1}.$ Therefore there are elements $P,P'$
of $R_{G}$ such that $\left(\begin{array}{ccc}
h'(H_{1}) & h'(H_{2}) & h'(H_{3})\end{array}\right)h'(\mathsf{M}_{123}^{\dagger})=h'(P)\left(\begin{array}{ccc}
1 & w & w^{2}\end{array}\right)$ and $\left(\begin{array}{ccc}
h'(H_{1}') & h'(H_{2}') & h'(H_{3}')\end{array}\right)h'(\mathsf{M}_{123}^{\dagger})=h'(P')\left(\begin{array}{ccc}
1 & w & w^{2}\end{array}\right)$. Hence, by (\ref{eq:M123}), we have 
\begin{equation}
h'(P)\left(\begin{array}{ccc}
1 & w & w^{2}\end{array}\right)=h'(P')\left(\begin{array}{ccc}
1 & w & w^{2}\end{array}\right)h'(\mathsf{A}).\label{eq:hP}
\end{equation}
Computing the l.h.s.~of (\ref{eq:M123}) explicitly and noting it
is quadratic with respect to $t_{1}$, we can check that $h'(P)$,
which is the $(1,1)$ entry of the l.h.s.~of (\ref{eq:M123}), is
an irreducible polynomial in $\tilde{S_{G}}$ without any monomials
divided by positive powers of $w$. Therefore, by (\ref{eq:hP}),
$h'(P')$ is a constant or the $(1,1)$ entry $n_{11}$ of $\left(\begin{array}{ccc}
1 & w & w^{2}\end{array}\right)h'(\mathsf{A})$ is a constant. Assume that $n_{11}$ is a constant. Then, since $\left(\begin{array}{ccc}
r(s_{1}') & r(s_{2}') & r(s_{3}')\end{array}\right)=\left(\begin{array}{ccc}
1 & w & w^{2}\end{array}\right)$, we see that $r(s_{1}')$ is a constants by (\ref{eq:ss'}). Then,
by Lemma \ref{lem:f1f2f3} below, we see that $r(s_{1}),r(s_{2}),r(s_{3})$
generates $\tilde{S}_{G}$ as an $R_{G}$-module, which implies that
$R_{G}[J_{D}^{-1}]$ is generated by $s_{1}$,~$s_{2}$,~$s_{3}$
as $R_{G}$-algebra. It remains to derive a contradiction assuming
$h'(P')$ is a constant. Note that the entries of $\left(\begin{array}{ccc}
H_{1}' & H_{2}' & H_{3}'\end{array}\right)\mathsf{M}_{123}^{\dagger}$ are elements of $J_{{\rm Sing}\,D}$ by Proposition \ref{prop:SingD},
where $J_{{\rm Sing}\,D}\subset R_{G}$ is the ideal of ${\rm Sing\,}D$.
Since ${\rm Ker}\,h'=J_{D}$ and $J_{D}\subset J_{{\rm Sing}\,D}$,
we see that $J_{{\rm Sing}\,D}$ would contain a constant, a contradiction
since ${\rm Sing}\,D_{{\rm }}$ is nonempty.

Let $\varphi\colon R_{G}[S_{1},S_{2},S_{3}]\to R_{G}[J_{D}^{-1}]$
be the ring homomorphism defined by $S_{i}\mapsto s_{i}\,(i=1,2,3)$
and $\widetilde{F}_{j}\,(1\leq j\leq6)$ elements of $R_{G}[S_{1},S_{2},S_{3}]$
obtained from $F_{j}$ by substituting $s_{i}$ with $S_{i}$. As
we saw above, $\varphi$ is surjective. It is clear that $\widetilde{F}_{j}\in\ker\varphi$.
Now we show that $\ker\varphi\subset(\widetilde{F}_{1},\dots,\widetilde{F}_{9})$,
which implies that $\varphi$ is injective and completes the proof
of the proposition. Let $G$ be any element of $\ker\varphi$. Using
$\widetilde{F}_{7},\widetilde{F}_{8},\widetilde{F}_{9}$, we may assume
that $G$ is linear with respect to $S_{2}$ and $S_{3}$. Thus we
may write $G=\alpha_{1}S_{1}^{k}+S_{1}^{k-1}(\alpha_{2}S_{2}+\alpha_{3}S_{3})+H$,
where $k\in\mathbb{N}$, $\alpha_{1},\alpha_{2},\alpha_{3}\in R_{G}$
and $H\in R_{G}[S_{1},S_{2},S_{3}]$ is of total degree at most $k-1$.
If $k=1$, then, by virtue of Proposition \ref{prop:The-first-syzygy},
we can show $G\in(\widetilde{F}_{1},\dots,\widetilde{F}_{9})$ almost
in the same way as in the proof of \cite[Prop.~2.3.2]{key-3}. If
$k\geq2$, then, by the proof of \cite[Prop.~2.3.2]{key-3} again,
we can show $G\in(\widetilde{F}_{1},\dots,\widetilde{F}_{9})$ by
induction on $k$.
\end{proof}
\begin{lem}
\label{lem:f1f2f3} Let $f_{1}$,~$f_{2}$,~$f_{3}$ be elements
of $\tilde{S}_{G}$ such that $\left(\begin{array}{ccc}
f_{1} & f_{2} & f_{3}\end{array}\right)h'(\mathsf{M})=\bm{0}$. It holds that $f_{2}=wf_{1}$ and $f_{3}=w^{2}f_{1}$.
\end{lem}

\begin{proof}
By the first column of $\left(\begin{array}{ccc}
f_{1} & f_{2} & f_{3}\end{array}\right)h'(\mathsf{M})=\bm{0}$, we have $(f_{1}w-f_{2})p_{1}+(f_{1}w^{2}-f_{3})p_{2}=0$. Thus there
is an element $A\in\tilde{S}_{G}$ such that $f_{1}w-f_{2}=p_{2}A$
and $f_{1}w^{2}-f_{3}=-p_{1}A$. Similarly, by the second column of
$\left(\begin{array}{ccc}
f_{1} & f_{2} & f_{3}\end{array}\right)h'(\mathsf{M})=\bm{0}$, we see that there is an element $B\in\tilde{S}_{G}$ such that $f_{1}w-f_{2}=p_{4}B$
and $f_{1}w^{2}-f_{3}=-p_{3}B$. Therefore we have $p_{2}A=p_{4}B$
and $p_{1}A=p_{3}B$. This finally implies that $A=B=0$ and hence
we have the assertion.
\end{proof}

\subsection{Gorensteinness of $R_{\Pi}$}
\begin{prop}
\label{prop:Gor}The affine coordinate ring $R_{\Pi}$ is Gorenstein.
\end{prop}

\begin{proof}
Since we can proceed almost in the same way as in \cite[the subsec.~2.4]{key-3},
we only give a sketch of the proof (again we basically follow \cite{key-3}
but the origin of the arguments is in \cite{key-2}). 

To show $R_{\Pi}$ is Gorenstein, it is enough to prove $R_{\Pi}/(s_{1})$
is Gorenstein. Since $s_{1}\in J_{D}^{-1}$ by the proof of Proposition
\ref{prop:Unproj}, it defines an $R_{G}$-homomorphism $J_{D}\to R_{G}$,
which we also denote by $s_{1}$. By \cite[Lem.~1.1]{key-1}, we may
assume that $s_{1}$ is injective. By this possible replacement of
$s_{1}$, the polynomials $F_{1},\dots,F_{9}$ can be changed but
this is harmless since still $F_{1},\dots,F_{6}$ are linear and $F_{7},F_{8},F_{9}$
are quadratic with respect to $s_{1},s_{2},s_{3}$, and the linear
parts of $F_{1},\dots,F_{6}$ and the quadratic parts of $F_{7},F_{8},F_{9}$
are unchanged. Let $J_{N}$ be the image of the injection $s_{1}\colon J_{D}\to R_{G}$,
where $N$ denotes the closed subscheme of $\{G=0\}$ defined by $J_{N}$.
Then, following the proofs of \cite[Lem.~2.4.1 and Lem.~2.4.2]{key-3},
we can show the codimension of $N$ in $\{G=0\}$ is one and the following
isomorphisms of $R_{G}$-modules hold:
\begin{equation}
R_{\Pi}/(s_{1})\simeq{\rm Hom}_{R_{G}}(J_{D},R_{G})/\langle s_{1}\rangle\simeq{\rm Hom}_{R_{G}}(J_{N},R_{G})/\langle i_{N}\rangle\simeq\omega_{N},\label{eq:isoms}
\end{equation}
where $i_{N}\colon J_{N}\to R_{G}$ is the standard injective $R_{G}$-homomorphism
(we also note that $R_{\Pi}/(s_{1})$ is a quotient ring by an ideal
while ${\rm Hom}_{R_{G}}(J_{D},R_{G})/\langle s_{1}\rangle$ and ${\rm Hom}(J_{N},R_{G})/\langle i_{N}\rangle$
are quotient $R_{G}$-modules by sub $R_{G}$-modules). We add one
remark; to get the first isomorphism, we use the equations $F_{7},F_{8},F_{9}$,
with which we can show any element of $R_{\Pi}/(s_{1})$ has at most
linear terms with respect to $s_{2},s_{3}$ (see \cite[Prop.~2.10]{key-2}
for details). By (\ref{eq:isoms}), we see that $R_{\Pi}/(s_{1})$
is finite as an $R_{G}$-module by (\ref{eq:isoms}) since so is $R_{G}/J_{N}$
and hence so is $\omega_{N}$. Now, following the proof of \cite[Lem.~2.4.3]{key-3},
we can show that ${\rm depth}_{R_{G}}\omega_{N}=\dim\{G=0\}-1$ (we
note that, since $\nu_{D}\colon\tilde{D}\to D$ is isomorphic in codimension
one by Proposition \ref{prop:SingD}, it holds that $\omega_{D}\simeq\mathcal{O}_{\tilde{D}}$
as an $\mathcal{O}_{D}$-module. Hence $\omega_{D}$ is a Cohen-Macaulay
$\mathcal{O}_{D}$-module). Thus we have $\dim\omega_{N}\geq\dim\{G=0\}-1$.
On the other hand, since the codimension of $N$ in $\{G=0\}$ is
one, we have $\dim\omega_{N}\leq\dim\{G=0\}-1$. Thus it holds that
$\dim\omega_{N}=\dim\{G=0\}-1$, hence ${\rm depth}_{R_{G}}\omega_{N}=\dim\omega_{N}$.
This implies that $R_{\Pi}/(s_{1})$ is a Cohen-Macaulay $R_{G}$-module
by (\ref{eq:isoms}) and hence $R_{\Pi}/(s_{1})$ is a Cohen-Macaulay
ring. To show $R_{\Pi}/(s_{1})$ is Gorenstein, it remains to show
that $\omega_{R_{\Pi}/(s_{1})}$ exists and is isomorphic to $R_{\Pi}/(s_{1})$.
Since $R_{\Pi}/(s_{1})$ is finite as an $R_{G}$-module, $\omega_{R_{\Pi}/(s_{1})}$
exists and $\omega_{R_{\Pi}/(s_{1})}\simeq{\rm Ext^{1}}_{R_{G}}(R_{\Pi}/(s_{1}),R_{G})$
by \cite[Thm.~3.3.7 (b)]{key-4}. Therefore, to show that $\omega_{R_{\Pi}/(s_{1})}\simeq R_{\Pi}/(s_{1})$,
it suffices to check that $\omega_{N}\simeq{\rm Ext^{1}}_{R_{G}}(\omega_{N},R_{G})$
by (\ref{eq:isoms}). The proof of this fact is the same as that of
\cite[Lem.~2.4.4]{key-3} except checking $\dim\,{\rm coker\,}h_{3}=\dim R_{G}-3$,
where $h_{3}$ is the $R_{G}$-homomorphism $R_{D}\to\omega_{D}$
defined exactly in the same way as in \cite[Lem.~2.4.4]{key-3}. Noting
$\omega_{D}\simeq\omega_{\tilde{D}}\simeq\tilde{S}_{G}$, we can check
that $h_{3}$ can be identifies with the standard $R_{G}$-homomorphism
$R_{D}\to\tilde{S}_{G}$, whose cokernel is supported in ${\rm Sing}\,D$.
Therefore, by Proposition \ref{prop:SingD}, we obtain $\dim{\rm coker}\,h_{3}=\dim R_{G}-3$
as desired.
\end{proof}

\section{${\rm GL_{2}}$- and $(\mathbb{C}^{*})^{4}$-actions on $\Pi_{\mathbb{A}}^{15}$\label{sec:--and--actions}}
\begin{prop}
\label{prop:GL}For any element $g=\left(\begin{array}{cc}
a & b\\
c & d
\end{array}\right)$ of ${\rm GL}_{2}$, we set {\small{}
\[
\hat{g}=(1/\det g)\left(\begin{array}{ccc}
d^{2} & -cd & c^{2}\\
-2bd & bc+ad & -2ac\\
b^{2} & -ab & a^{2}
\end{array}\right).
\]
}The group ${\rm GL_{2}}$ acts on the affine variety $\Pi_{\mathbb{A}}^{15}$
by the following rules: for any $g\in{\rm GL_{2}}$, we set 

{\small{}
\begin{align*}
 & \left(\begin{array}{cc}
p_{1} & p_{3}\\
p_{2} & p_{4}\\
u & v
\end{array}\right)\mapsto\left(\begin{array}{cc}
p_{1} & p_{3}\\
p_{2} & p_{4}\\
u & v
\end{array}\right)g,\,\left(\begin{array}{c}
s_{1}\\
s_{2}\\
s_{3}
\end{array}\right)\mapsto\det g\left(\begin{array}{c}
s_{1}\\
s_{2}\\
s_{3}
\end{array}\right),\\
 & \left(\begin{array}{cc}
L_{123} & L_{125}\\
L_{124} & L_{126}
\end{array}\right)\mapsto g^{\dagger}\left(\begin{array}{cc}
L_{123} & L_{125}\\
L_{124} & L_{126}
\end{array}\right),\\
 & \left(\begin{array}{ccc}
L_{136} & L_{245} & L_{246}\\
-L_{135} & -L_{136} & -L_{245}
\end{array}\right)\mapsto\empty^{t}\!\!g\left(\begin{array}{ccc}
L_{136} & L_{245} & L_{246}\\
-L_{135} & -L_{136} & -L_{245}
\end{array}\right)\hat{g},\\
 & t_{1}\mapsto t_{1},\,t_{2}\mapsto t_{2},
\end{align*}
}where $g^{\dagger}$ is the adjoint matrix of $g$.

Moreover, the conditions on the weights of coordinates in the subsection
$\ref{subsec:Weights-for-variables}$ determines a $(\mathbb{C}^{*})^{4}$-action
on $\Pi_{\mathbb{A}}^{15}$.
\end{prop}

\begin{proof}
By direct computations, we may check that {\small{}
\begin{align*}
 & \left(\begin{array}{ccc}
F_{1} & F_{3} & F_{5}\\
F_{2} & F_{4} & F_{6}
\end{array}\right)\mapsto(\det g)\,\empty^{t}\!\!g\left(\begin{array}{ccc}
F_{1} & F_{3} & F_{5}\\
F_{2} & F_{4} & F_{6}
\end{array}\right),\\
 & \left(\begin{array}{c}
F_{7}\\
F_{8}\\
F_{9}
\end{array}\right)\mapsto(\det g)^{2}\left(\begin{array}{c}
F_{7}\\
F_{8}\\
F_{9}
\end{array}\right)
\end{align*}
}for $g\in{\rm GL_{2}}${\small{}.} Therefore the above rules defines
a ${\rm GL_{2}}$-action on $\Pi_{\mathbb{A}}^{15}$.

The assertion on the $(\mathbb{C^{*}})^{4}$-action follows since
we may choose $w(p_{3})$, $w(u)$, $w(v)$, $w(G)$ freely.
\end{proof}
\begin{rem}
The $GL_{2}$-action on $\Pi_{\mathbb{A}}^{15}$ is more visible by
the following presentations of $F_{1}$--$F_{6}$.

We set 

{\small{}
\begin{align*}
 & \mathsf{L}:=\left(\begin{array}{ccc}
L_{136} & L_{245} & L_{246}\\
-L_{135} & -L_{136} & -L_{245}
\end{array}\right),\\
 & \mathsf{U}:=\left(\begin{array}{c}
U_{1}\\
U_{2}\\
U_{3}
\end{array}\right)=\left(\begin{array}{c}
p_{1}^{2}+t_{1}p_{2}^{2}+p_{2}u\\
2(p_{1}p_{3}+t_{1}p_{2}p_{4})+(p_{4}u+p_{2}v)\\
p_{3}^{2}+t_{1}p_{4}^{2}+p_{4}v
\end{array}\right),\\
 & \mathsf{I}:=\left(\begin{array}{c}
I_{1}\\
I_{2}\\
I_{3}
\end{array}\right)=\left(\begin{array}{c}
t_{2}p_{2}^{2}+p_{1}u\\
2t_{2}p_{2}p_{4}+(p_{3}u+p_{1}v)\\
t_{2}p_{4}^{2}+p_{3}v
\end{array}\right),\\
 & \mathsf{A}:=\left(\begin{array}{c}
A_{1}\\
A_{2}\\
A_{3}
\end{array}\right)=\left(\begin{array}{c}
t_{1}p_{2}u-t_{2}p_{1}p_{2}+u^{2}\\
t_{1}(p_{4}u+p_{2}v)-t_{2}(p_{2}p_{3}+p_{1}p_{4})\\
t_{1}p_{4}v-t_{2}p_{3}p_{4}+v^{2}
\end{array}+2uv\right),
\end{align*}
\begin{align*}
 & \bm{{v}}_{1}:=(-p_{2}p_{3}+p_{1}p_{4})\left(\begin{array}{c}
L_{126}\\
-L_{125}
\end{array}\right)+\mathsf{LU},\\
 & \bm{{v}}_{2}:=(-p_{2}p_{3}+p_{1}p_{4})\left(\begin{array}{c}
L_{124}\\
-L_{123}
\end{array}\right)+\mathsf{LI},\\
 & \bm{{v}}_{3}:=-\left(\begin{array}{cc}
U_{2} & -U_{1}\\
U_{3} & -U_{2}
\end{array}\right)\left(\begin{array}{c}
L_{124}\\
-L_{123}
\end{array}\right)+\left(\begin{array}{cc}
I_{2} & -I_{1}\\
I_{3} & -I_{2}
\end{array}\right)\left(\begin{array}{c}
L_{126}\\
-L_{125}
\end{array}\right)+\mathsf{LA}\\
 & \quad+1/2(-p_{3}u+p_{1}v)\left(\begin{array}{c}
L_{126}\\
-L_{125}
\end{array}\right)+1/2(p_{4}u-p_{2}v)\left(\begin{array}{c}
L_{124}\\
-L_{123}
\end{array}\right).
\end{align*}
}Then we have

{\small{}
\begin{align*}
 & \left(\begin{array}{ccc}
F_{1} & F_{3} & F_{5}\\
F_{2} & F_{4} & F_{6}
\end{array}\right)=\\
 & \quad\left(\begin{array}{ccc}
u & -p_{1} & -p_{2}\\
v & -p_{3} & -p_{4}
\end{array}\right)\left(\begin{array}{ccc}
s_{1} & s_{2} & s_{3}\\
s_{2} & s_{3} & t_{2}s_{1}-t_{1}s_{2}\\
s_{3} & t_{2}s_{1}-t_{1}s_{2} & t_{2}s_{2}-t_{1}s_{3}
\end{array}\right)-(\bm{{v}}_{1}\bm{\,{v}}_{2}\,\bm{{v}}_{3}).
\end{align*}
} It is desirable that $F_{7}$--$F_{9}$ are also interpreted in
this context. 
\end{rem}

\section{Singular locus of $\Pi_{\mathbb{A}}^{15}$\label{sec:Singular-locus-Pi}}


Let $\mathbb{A}_{\Pi}^{10}$ be the affine $10$-space with the coordinates
$t_{1}$,~$t_{2}$,~$L_{123}$,~$L_{124}$,~$L_{125}$,~$L_{126}$,
$L_{135}$,~$L_{136}$,~$L_{245}$, and $L_{246}$. In this section,
we consider $\mathbb{A}_{\Pi}^{10}$ is contained in the ambient affine
space of $\Pi_{\mathbb{A}}^{15}$. Then we see that $\mathbb{A}_{\Pi}^{10}\subset\Pi_{\mathbb{A}}^{15}$.
\begin{prop}
\label{prop:SingPi}${\rm {Sing}\,\Pi_{\mathbb{A}}^{15}}$ is the
union of the affine space $\mathbb{A}_{\Pi}^{10}$ and a locally closed
subset $S$ of codimension seven in $\Pi_{\mathbb{A}}^{15}$. Along
$S$, $\Pi_{\mathbb{A}}^{15}$ has $c(G(2,5))$-singularities, where
we call a singularity isomorphic to the vertex of the cone over $G(2,5)$
a $c(G(2,5))$-singularity. Moreover, it holds that 

{\small{}
\begin{align*}
 & S\cap\{p_{2}p_{4}\not=0\}=\\
 & \{\,p_{2}p_{4}\not=0,\\
 & p_{1}=(p_{2}p_{3})/p_{4},\\
 & u=p_{2}/p_{4}v,\,v=-(p_{3}^{2}+t_{1}p_{4}^{2})/p_{4},\,t_{2}=-p_{3}v/p_{4}^{2},\\
 & s_{1}=-(L_{136}p_{2}^{2}+2L_{245}p_{2}p_{4}+L_{246}p_{4}^{2})/p_{2},\,s_{2}=(p_{3}/p_{4})s_{1},\,s_{3}=-((2p_{3}^{2}+t_{1}p_{4}^{2})/p_{4}^{2})s_{1},\\
 & L_{123}=-(L_{126}p_{2}p_{3}p_{4}^{2}+L_{136}(3p_{3}^{2}+2t_{1}p_{4}^{2})p_{2}^{2}+3L_{245}(p_{3}^{2}+t_{1}p_{4}^{2})p_{2}p_{4}+L_{246}t_{1}p_{4}^{4})/(p_{2}^{2}p_{4}^{2}),\\
 & L_{124}=(L_{126}p_{2}p_{3}p_{4}-L_{245}(3p_{3}^{2}+t_{1}p_{4}^{2})p_{2}-L_{246}(3p_{3}^{2}+t_{1}p_{4}^{2})p_{4})/(p_{2}p_{4}^{2}),\\
 & L_{125}=(-L_{126}p_{2}p_{4}^{2}+3L_{136}p_{2}^{2}p_{3}+6L_{245}p_{2}p_{3}p_{4}+3L_{246}p_{3}p_{4}^{2})/(p_{2}^{2}p_{4}),\\
 & L_{135}=-(p_{4}(3L_{136}p_{2}^{2}+3L_{245}p_{2}p_{4}+L_{246}p_{4}^{2}))/p_{2}^{3}\}
\end{align*}
}{\small\par}
\end{prop}

In particular, $\Pi_{\mathbb{A}}^{15}$ is normal.
\begin{proof}
Assume that $p_{2}p_{4}\not=0$. Then, we can solve the equation $F_{2}=0$
with respect to $s_{3}$, and hence we may erase $s_{3}$. Now we
consider the following coordinate change: The coordinates $p_{2},\,p_{3},\,p_{4},\,t_{1},\,L_{126},\,L_{136},\,L_{245},\,L_{246}$
are unchanged while the coordinates $L_{123}$, $L_{124}$, $s_{2}$,
$t_{2}$, $L_{125}$, $L_{135}$, $v$, $s_{1}$, $u$, $p_{1}$ are
changed to 

{\small{}
\begin{align*}
 & m_{12}=9/(p_{2}p_{4}^{3})(L_{123}p_{2}p_{4}^{2}+L_{126}p_{3}p_{4}^{2}+p_{3}p_{4}s_{2}+L_{135}p_{2}p_{4}u\\
 & \qquad+L_{136}(p_{2}p_{3}^{2}+p_{1}p_{3}p_{4}+p_{2}p_{4}v)+2L_{245}p_{3}^{2}p_{4}),\\
 & m_{13}=-(3/(p_{2}^{2}p_{4}^{2}))(L_{135}(p_{1}^{2}+t_{1}p_{2}^{2})p_{2}^{2}-L_{125}(p_{2}p_{3}-p_{1}p_{4})p_{2}^{2}-L_{126}p_{2}^{2}p_{3}p_{4}\\
 & \qquad+L_{246}(p_{2}p_{3}+p_{1}p_{4})p_{3}p_{4}+L_{124}p_{2}^{2}p_{4}^{2}+p_{2}(-p_{2}p_{3}+p_{1}p_{4})s_{2}+L_{135}p_{2}^{2}u\\
 & \qquad+p_{2}^{2}s_{1}v+L_{136}(2p_{1}p_{3}+2t_{1}p_{2}p_{4}+2p_{4}u+p_{2}v)p_{2}^{2}\\
 & \qquad+L_{245}(p_{2}p_{3}^{2}+2p_{1}p_{3}p_{4}+t_{1}p_{2}p_{4}^{2}+2p_{2}p_{4}v)p_{2}),\\
 & m_{14}=3/(p_{2}^{2}p_{4})(p_{2}p_{4}s_{2}+L_{136}p_{2}^{2}p_{3}+2L_{245}p_{2}p_{3}p_{4}+L_{246}p_{3}p_{4}^{2}),\\
 & m_{15}=-3(p_{3}u+p_{2}p_{4}t_{2})/p_{4},\\
 & m_{23}=3/(p_{2}p_{4}^{2})(-(p_{2}p_{3}+p_{1}p_{4})s_{1}-p_{2}p_{4}s_{2}-L_{125}p_{2}^{2}p_{4}-L_{126}p_{2}p_{4}^{2}\\
 & \qquad+L_{136}(-p_{2}p_{3}+p_{1}p_{4})p_{2}-L_{245}(p_{2}p_{3}+p_{1}p_{4})p_{4}),\\
 & m_{24}=3/(p_{2}p_{4})(-L_{135}p_{2}^{2}+p_{4}s_{1}-2L_{136}p_{2}p_{4}-L_{245}p_{4}^{2}),\\
 & m_{25}=-(3p_{2}/p_{4}^{2})(p_{3}^{2}+t_{1}p_{4}^{2}+p_{4}v),\\
 & m_{34}=(p_{2}s_{1}+L_{136}p_{2}^{2}+2L_{245}p_{2}p_{4}+L_{246}p_{4}^{2})/p_{2},\\
 & m_{35}=-(p_{1}^{2}+t_{1}p_{2}^{2}+p_{2}u)\\
 & m_{45}=-p_{2}p_{3}+p_{1}p_{4}.
\end{align*}
}Indeed, these define a coordinate change since these equations can
be solve regularly with respect to $L_{123}$, $L_{124}$, $s_{2}$,
$t_{2}$, $L_{125}$, $L_{135}$, $v$, $s_{1}$, $u$, $p_{1}$.
Then, by direct computations for this coordinate change, we see that
$\Pi_{\mathbb{A}}^{15}\cap\{p_{2}p_{4}\not=0\}$ is isomorphic to
the variety defined by the five $4\times4$ Pfaffians of the skew-symmetric
matrix {\small{}
\[
\left(\begin{array}{ccccc}
0 & m_{12} & m_{13} & m_{14} & m_{15}\\
 & 0 & m_{23} & m_{24} & m_{25}\\
 &  & 0 & m_{34} & m_{35}\\
 &  &  & 0 & m_{45}\\
 &  &  &  & 0
\end{array}\right)
\]
}in the open subset $\{p_{2}p_{4}\not=0\}$ of the affine space $\mathbb{A}^{18}$
with the coordinates $m_{ij}\,(1\leq i<j\leq5),$ and $p_{2},\,p_{3},\,p_{4},\,t_{1},\,L_{126},\,L_{136},\,L_{245},\,L_{246}$.
Therefore $\Pi_{\mathbb{A}}^{15}\cap\{p_{2}p_{4}\not=0\}$ is singular
along the locus corresponding to $\{m_{ij}=0\,(1\leq i<j\leq5)\}$
and has $c(G(2,5))$-singularities there. We may verify $\{m_{ij}=0\,(1\leq i<j\leq5)\}\cap\{p_{2}p_{4}\not=0\}=S\cap\{p_{2}p_{4}\not=0\}$
as in the statement of the proposition. 

Since points with $p_{2}\not=0$ or $p_{4}\not=0$ can be transformed
by the ${\rm GL}_{2}$-action on $\Pi_{\mathbb{A}}^{15}$ to points
with $p_{2}p_{4}\not=0$ (Proposition \ref{prop:GL}), we see that
$\Pi_{\mathbb{A}}^{15}$ has also $c(G(2,5))$-singularities on the
locus $\{p_{2}=0,p_{4}\not=0\}\cup\{p_{2}\not=0,p_{4}=0\}$.

Finally, we check the singularities of $\Pi_{\mathbb{A}}^{15}$ on
$\{p_{2}=p_{4}=0\}$ and show that ${\rm Sing\,}\Pi_{\mathbb{A}}^{15}\cap\{p_{2}=p_{4}=0\}=\mathbb{A}_{\Pi}^{10}$.
Before studying singularities of $\Pi_{\mathbb{A}}^{15}$ on $\{p_{2}=p_{4}=0\}$,
we show the smoothness of $\Pi_{\mathbb{A}}^{15}$ on several open
subsets. If $p_{1}p_{4}-p_{2}p_{3}\not=0$, then we may eliminate
$L_{123},L_{124},L_{125},L_{126}$ from the equations of $\Pi_{\mathbb{A}}^{15}$,
thus $\Pi_{\mathbb{A}}^{15}$ is smooth on the open subset $\{p_{1}p_{4}-p_{2}p_{3}\not=0\}.$
If $U_{1}=p_{1}^{2}+t_{1}p_{2}^{2}+p_{2}u\not=0$, then we may eliminate
$L_{123},L_{135},L{}_{136}$ from the equations of $\Pi_{\mathbb{A}}^{15}$,
so $\Pi_{\mathbb{A}}^{15}$ is isomorphic to a hypersurface on the
open subset $\{U_{1}\not=0\}$. We can verify that this hypersurface
is smooth. Therefore $\Pi_{\mathbb{A}}^{15}$ is smooth on the open
subset $\{U_{1}\not=0\}.$ By the ${\rm GL}_{2}$-symmetry of $\Pi_{\mathbb{A}}^{15}$,
we see that $\Pi_{\mathbb{A}}^{15}$ is smooth on the open subset
$\{U_{3}\not=0\}.$ By the ${\rm GL}_{2}$-action on $\Pi_{\mathbb{A}}^{15}$
again, we also see that if $U_{1}=U_{3}=0$ and $U_{2}\not=0$, then
$\Pi_{\mathbb{A}}^{15}$ is smooth. Therefore $\Pi_{\mathbb{A}}^{15}$
is smooth outside $\{\mathsf{U=0}\}$. Similarly, we can show $\Pi_{\mathbb{A}}^{15}$
is smooth outside $\{\mathsf{I=0}\}$. 

Therefore it holds that ${\rm Sing}\,\Pi_{\mathbb{A}}^{15}\subset\{\mathsf{U=I=0},p_{1}p_{4}-p_{2}p_{3}=0\}$
by the above considerations.

Now we assume that $p_{2}=p_{4}=0$. Then we can verify that $\{\mathsf{U}=\mathsf{I}=0,p_{1}p_{4}-p_{2}p_{3}=0,p_{2}=p_{4}=0\}\cap\Pi_{\mathbb{A}}^{15}=(\Pi_{\mathbb{A}}^{15}\cap\{p_{1}=p_{2}=p_{3}=p_{4}=0\})\cap(\{u=v=0\}\cup\{s_{1}=s_{2}=0\}).$
We can calculate the Zariski tangent space at any point of these loci
outside $\mathbb{A}_{\Pi}^{10}$, and then can show that $\Pi_{\mathbb{A}}^{15}$
is smooth there. For example, let $\mathsf{p}$ be a point of $(\Pi_{\mathbb{A}}^{15}\cap\{p_{1}=p_{2}=p_{3}=p_{4}=u=v=0\})\setminus\mathbb{A}_{\Pi}^{10}$.
Then we have one of $s_{1},s_{2},s_{3}\not=0$. Assume that $s_{3}\not=0$.
Then, by the polynomial $F_{9}$, we have also $s_{2}\not=0$. Hence,
by the polynomials $F_{7},F_{8},F_{9}$, we have $s_{1}=s_{2}^{2}/s_{3},t_{2}=(t_{1}s_{2}^{2}s_{3}+s_{3}^{3})/s_{2}^{3}$.
Assume that $s_{3}=0$. Then, by the polynomial $F_{7}$, we have
also $s_{2}=0$. Hence, we obtain $(\Pi_{\mathbb{A}}^{15}\cap\{p_{1}=p_{2}=p_{3}=p_{4}=u=v=0\})\setminus\mathbb{A}_{\Pi}^{10}=\{p_{1}=p_{2}=p_{3}=p_{4}=u=v=0\}\cap(\{s_{2}s_{3}\not=0,s_{1}=s_{2}^{2}/s_{3},t_{2}=(t_{1}s_{2}^{2}s_{3}+s_{3}^{3})/s_{2}^{3}\}\cup\{s_{1}\not=0,s_{2}=s_{3}=0\})$.
Using this descriptions, we can easily calculate the Zariski tangent
space of $\Pi_{\mathbb{A}}^{15}$ at $\mathsf{p}$ using the polynomials
$F_{1}$--$F_{9}$.

By the dimension of the singular locus, $\Pi_{\mathbb{A}}^{15}$ satisfies
the $R_{1}$-condition and also satisfies the $S_{2}$-condition by
Proposition \ref{prop:Gor}. Therefore $\Pi_{\mathbb{A}}^{15}$ is
normal.
\end{proof}

\section{Factoriality of the affine coordinate ring of $\Pi_{\mathbb{A}}^{15}$\label{sec:Factoriality-of-the}}

In this section, we show that the affine coordinate ring $R_{\Pi}$
of $\Pi_{\mathbb{A}}^{15}$ is a UFD, and give applications of this
result. We use the following polynomial $b$ frequently:
\begin{equation}
b:=p_{3}^{2}+t_{1}p_{4}^{2}+p_{4}v,\label{eq:boundary}
\end{equation}
which is one of the $2\times2$ minors of the matrix $\mathsf{M}$.
We denote by $\bar{b}$ the image of $b$ in $R_{\Pi}$. In the proof
of Lemma \ref{lem:prime} and Proposition \ref{prop:UFD} below, we
closely follows the arguments in \cite[the proof of Prop.~2.4]{key-5}.
\begin{lem}
\label{lem:prime}The element of $\bar{b}\in R_{\Pi}$ is a prime
element, equivalently, $\Pi_{\mathbb{A}}^{15}\cap\{b=0\}$ is irreducible
and reduced.
\end{lem}

\begin{proof}
First we show that the assertion follows from the normality of $\Pi_{\mathbb{A}}^{15}\cap\{b=0\}.$
Note that the polynomial ring $S_{\Pi}$ is positively graded with
some weights of variables such that $I_{\Pi}$ is homogeneous and
$b$ is semi-invariant (see the subsection \ref{subsec:Weights-for-variables}).
Then the ring $R_{\Pi}$ is also graded with these weights. In this
situation, $\bar{b}$ defines an ample divisor on the corresponding
weighted projectivization $\Pi_{\mathbb{P}}^{14}$ of $\Pi_{\mathbb{A}}^{15}$.
Note that $\Pi_{\mathbb{P}}^{14}$ is irreducible and normal since
so is $\Pi_{\mathbb{A}}^{15}$ by Propositions \ref{prop:Unproj}
and \ref{prop:Gor}, and $\Pi_{\mathbb{P}}^{14}$ is a geometric $\mathbb{C^{*}}$-quotient
of $\Pi_{\mathbb{A}}^{15}$. By a general property of an ample divisor
on an irreducible projective normal variety, the Cartier divisor $\Pi_{\mathbb{P}}^{14}\cap\{b=0\}$
is connected and hence so is $\Pi_{\mathbb{A}}^{15}\cap\{b=0\}$.
Therefore the normality of $\Pi_{\mathbb{A}}^{15}\cap\{b=0\}$ implies
that $\Pi_{\mathbb{A}}^{15}\cap\{b=0\}$ is irreducible and reduced.

Now we show that $\Pi_{\mathbb{A}}^{15}\cap\{b=0\}$ is normal. Since
it is Gorenstein by Proposition \ref{prop:Gor} (note that $\bar{b}$
is not a zero divisor by Proposition \ref{prop:Unproj}), we have
only to show that $\Pi_{\mathbb{A}}^{15}\cap\{b=0\}$ satisfies the
$R_{1}$-condition. Assume that $p_{1}p_{4}\not=0$. Then, by the
proof of Proposition \ref{prop:SingPi}, $\Pi_{\mathbb{A}}^{15}\cap\{b=0\}$
is isomorphic to the divisor $\{m_{25}=0\}$ in the cone defined by
the five $4\times4$ Pfaffians of the skew-symmetric matrix $(m_{ij})$
with the vertex $\{p_{2}p_{4}\not=0\}\subset\mathbb{A}(p_{2},p_{3},p_{4},t_{1},L_{126},L_{136},L_{245},L_{246})$.
Therefore it is easy to verify that it satisfies the $R_{1}$-condition.
Then we have only to check that ${\rm codim_{\Pi}}\left(\{p_{i}=0\}\cap{\rm Sing}(\Pi_{\mathbb{A}}^{15}\cap\{b=0\})\right)\geq7$
for $i=2,4$. For this, we note that, by the proof of Proposition
\ref{prop:SingPi}, $\Pi_{\mathbb{A}}^{15}$ is isomorphic to hypersurfaces
on the open sets $\Pi_{\mathbb{A}}^{15}\cap\{I_{1}\not=0\}$, $\Pi_{\mathbb{A}}^{15}\cap\{I_{3}\not=0\}$,
and $\Pi_{\mathbb{A}}^{15}\cap\{p_{1}p_{4}-p_{2}p_{3}\not=0\}$. Using
these descriptions, we may verify that $(\Pi_{\mathbb{A}}^{15}\cap\{b=0\})$
satisfies the $R_{1}$-condition on these open sets. Therefore it
suffices to show that ${\rm codim_{\Pi}}\left(\{p_{i}=0,\,I_{1}=0,\,I_{3}=0,\,p_{1}p_{4}-p_{2}p_{3}=0\}\cap{\rm Sing}(\Pi_{\mathbb{A}}^{15}\cap\{b=0\})\right)\geq7$
for $i=2,4$, which can be also verified by a straightforward calculations.
\end{proof}
In the proof of Lemma \ref{lem:prime}, we have proved the following:
\begin{cor}
\label{cor:boundarynormal.}$\Pi_{\mathbb{A}}^{15}\cap\{b=0\}$ is
normal.
\end{cor}

\begin{prop}
\label{prop:UFD}The affine coordinate ring $R_{\Pi}$ of $\Pi_{\mathbb{A}}^{15}$
is a UFD.
\end{prop}

\begin{proof}
As we saw in the proof of Lemma \ref{lem:prime}, $R_{\Pi}$ can be
positively graded. In this situation, by \cite[Prop.~7.4]{key-9},
$R_{\Pi}$ is UFD if and only if so is the localization $R_{\Pi,o}$
of $R_{\Pi}$ with respect to the maximal irrelevant ideal. Thus we
have only to prove that $R_{\Pi,o}$ is UFD. Abusing notation, we
denote $R_{\Pi,o}$ by $R_{\Pi}$. We already know that $R_{\Pi}$
is a domain by Proposition \ref{prop:Unproj}. Then, by Nagata's theorem
\cite[Thm.~20.2]{key-8}, it suffices to show that $\bar{b}$ is a
prime element, and the localization $(R_{\Pi})_{\bar{b}}$ is a UFD.
We already show the first assertion in Lemma \ref{lem:prime}. Thus
we have only to show the second assertion. We see that we may solve
regularly on the open subset $\Pi_{\mathbb{A}}^{15}\cap\{b\not=0\}$
the equation $F_{1}=F_{2}=F_{6}=0$ with respect to $L_{246},L_{245},L_{124}$,
namely, $L_{246},L_{245},L_{124}$ can be expressed as rational functions
with multiples of $b$ as denominators and with polynomials of other
variables as numerators. Replacing $L_{246},L_{245},L_{124}$ with
these expressions, the equations of $\Pi_{\mathbb{A}}^{15}$ on $\Pi_{\mathbb{A}}^{15}\cap\{b\not=0\}$
are reduced to one equation, of which we omit the explicit form since
it is lengthy. By the Jacobian criterion, we see that this is smooth.
Therefore $(R_{\Pi})_{\bar{b}}$ is a complete intersection (a hypersurface)
and its localizations with respect to all the prime ideals are regular.
This implies that $(R_{\Pi})_{\bar{b}}$ is a UFD by \cite[Cor.~3.10 and Thm.~3.13]{key-10}.
\end{proof}
To obtain information on $\Pi_{\mathbb{A}}^{14}$, we use the following: 
\begin{prop}
\label{prop:Pi14}There exists an isomorphism from $\Pi_{\mathbb{A}}^{15}\cap\{L_{246}\not=0\}$
to $\Pi_{\mathbb{A}}^{14}\times(\mathbb{A}^{1})^{*}$. Moreover, this
induces an isomorphism from $\Pi_{\mathbb{A}}^{15}\cap\{L_{246}\not=0,b\not=0\}$
to $(\Pi_{\mathbb{A}}^{14}\cap\{b\not=0\})\times(\mathbb{A}^{1})^{*}$.
\end{prop}

\begin{proof}
We can define a morphism $\iota_{1}\colon\Pi_{\mathbb{A}}^{15}\cap\{L_{246}\not=0\}\to\Pi_{\mathbb{A}}^{14}$
by the following coordinate change:
\[
\alpha\mapsto\begin{cases}
L_{246}^{-2}\alpha & :\alpha=s_{1},s_{2},s_{3}\\
\alpha & :\alpha=t_{1},t_{2}\\
L_{246}^{-1}\alpha & :\text{{otherwise}}.
\end{cases}
\]
Let $\iota_{2}\colon\Pi_{\mathbb{A}}^{15}\cap\{L_{246}\not=0\}\to(\mathbb{A}^{1})^{*}$
be the projection to $(\mathbb{A}^{1})^{*}$ with $L_{246}$ as coordinate.
We can show easily that the morphism $\iota_{1}\times\iota_{2}$ is
an isomorphism and the latter assertion holds.
\end{proof}
\begin{prop}
\label{prop:Pic1}Let $\Pi_{\mathbb{P}}^{13}$ be the weighted projectivization
of $\Pi_{\mathbb{\mathbb{A}}}^{14}$ with some positive weights of
coordinates, and $\Pi_{\mathbb{P}}^{14}$ the weighted cone over $\Pi_{\mathbb{P}}^{13}$
with a weight one coordinate added. It holds that

\vspace{3pt}

\noindent$(1)$ any prime Weil divisors on $\Pi_{\mathbb{P}}^{13}$
and $\Pi_{\mathbb{P}}^{14}$ are the intersections between $\Pi_{\mathbb{P}}^{13}$
and $\Pi_{\mathbb{P}}^{14}$ respectively, and weighted hypersurfaces.
In particular, $\Pi_{\mathbb{P}}^{13}$ and $\Pi_{\mathbb{P}}^{14}$
are $\mathbb{Q}$-factorial and have Picard number one, and

\noindent$(2)$ let $X$ be a quasi-smooth threefold such that $X$
is a codimension $10$ or $11$ weighted complete intersection in
$\Pi_{\mathbb{P}}^{13}$ or $\Pi_{\mathbb{P}}^{14}$ respectively.
Assume moreover that $X\cap\{b=0\}$ is a prime divisor. Then any
prime Weil divisor on $X$ is the intersection between $X$ and a
weighted hypersurface. In particular, $X$ is $\mathbb{Q}$-factorial
and has Picard number one.
\end{prop}

\begin{proof}
The proofs of \cite[Cor.~2.5 and 2.6, Prop.~3.3]{key-5} work almost
verbatim using Propositions \ref{prop:UFD} and \ref{prop:Pi14},
so we omit the proof here.
\end{proof}

\section{Graded $9\times16$ free resolution\label{sec:Graded--resolution}}

\subsection{Weights for variables and equations\label{subsec:Weights-for-variables}}

We assign weights for variables of the polynomial ring $S_{\Pi}$
such that all the equations $F_{1}$--$F_{9}$ are homogeneous. Moreover,
we assume that all the variables are not zero allowing some of them
are constants. Then it is easy to derive the following relations between
the weights of variables of $S_{\Pi}$:

\begin{align*}
 & {\rm \mathbf{Weights\,for\,variables}}\\
 & w(p_{1})=w(p_{3})+w(u)-w(v),\,w(p_{2})=2w(p_{3})+w(u)-2w(v),\,w(p_{4})=2w(p_{3})-w(v),\\
 & w(s_{1})=w(G)-(w(u)+3w(v)-2w(p_{3})),\,w(s_{2})=w(G)-(w(u)+2w(v)-w(p_{3})),\\
 & w(s_{3})=w(G)-(w(u)+w(v)),\\
 & w(t_{1})=2(w(v)-w(p_{3})),\,w(t_{2})=3(w(v)-w(p_{3})),\\
 & w(L_{123})=w(G)-(2w(p_{3})+2w(u)-w(v)),\,w(L_{124})=w(G)-(2w(p_{3})+w(u)),\\
 & w(L_{125})=w(G)-(w(p_{3})+2w(u)),\,w(L_{126})=w(G)-(w(p_{3})+w(u)+w(v)),\\
 & w(L_{135})=w(G)-3w(u),\,w(L_{136})=w(G)-(2w(u)+w(v)),\\
 & w(L_{245})=w(G)-(w(u)+2w(v)),\,w(L_{246})=w(G)-3w(v),\\
 & {\rm \mathbf{Weights\,for\,equations}}\\
 & d_{1}:=w(F_{1})=w(G)-(3w(v)-2w(p_{3})),\,d_{2}:=w(F_{2})=w(G)-(w(u)+2w(v)-2w(p_{3})),\\
 & d_{3}:=w(F_{3})=w(G)-(2w(v)-w(p_{3})),\,d_{4}:=w(F_{4})=w(G)-(w(u)+w(v)-w(p_{3})),\\
 & d_{5}:=w(F_{5})=w(G)-w(v),\,d_{6}:=w(F_{6})=w(G)-w(u),\\
 & d_{7}:=w(F_{7})=2(w(G)-w(u)-2w(v)+w(p_{3})),\\
 & d_{8}:=w(F_{8})=2w(G)-(2w(u)+3w(v)-w(p_{3})),\\
 & d_{9}:=w(F_{9})=2\big(w(G)-(w(u)+w(v))\big).
\end{align*}

\subsection{Graded $9\times16$ free resolution}

we set
\[
\delta:=4w(G)-(3w(u)+6w(v)-3w(p_{3})).
\]
Then we observe that the summation of all the weights of the equations
of $\Pi_{\mathbb{A}}^{15}$ is equal to

\[
\sum_{i=1}^{9}d_{i}=3\delta.
\]

Using \cite{key-18}, we obtain the following:
\begin{prop}
\label{prop:916}The following assertions hold:

\vspace{3pt}

\noindent$(1)$ The affine coordinate ring $R_{\Pi}$ of $\Pi_{\mathbb{A}}^{15}$
has the following so-called $9\times16$ $S_{\Pi}$-free resolution
which is graded with respect to the weights given as in the subsection
$\ref{subsec:Weights-for-variables}$:
\[
0\leftarrow R_{\Pi}\leftarrow P_{0}\leftarrow P_{1}\leftarrow P_{2}\leftarrow P_{3}\leftarrow P_{4}\leftarrow0,
\]
where {\small{}
\begin{align*}
 & P_{0}=S_{\Pi},P_{1}=\oplus_{i=1}^{9}S_{\Pi}(-d_{i}),\\
 & P_{2}=S_{\Pi}(-(d_{6}+w(p_{2})))\oplus S_{\Pi}(-(d_{6}+w(p_{1})))\oplus S_{\Pi}(-(d_{8}+w(p_{2})))\oplus S_{\Pi}(-(d_{8}+w(p_{4})))\\
 & \oplus S_{\Pi}(-(d_{8}+w(p_{1})))^{\oplus2}\oplus S_{\Pi}(-(d_{8}+w(p_{3})))^{\oplus2}\\
 & \oplus S_{\Pi}(-\delta+(d_{8}+w(p_{1})))^{\oplus2}\oplus S_{\Pi}(-\delta+(d_{8}+w(p_{3})))^{\oplus2}\\
 & \oplus S_{\Pi}(-\delta+(d_{6}+w(p_{2})))\oplus S_{\Pi}(-\delta+(d_{6}+w(p_{1})))\oplus S_{\Pi}(-\delta+(d_{8}+w(p_{2})))\oplus S_{\Pi}(-\delta+(d_{8}+w(p_{4}))),\\
 & P_{3}=\oplus_{i=1}^{9}S_{\Pi}(-\delta+d_{i}),P_{4}=S_{\Pi}(-\delta).
\end{align*}
}{\small\par}

Moreover, we have 
\[
K_{\Pi_{\mathbb{P}}^{14}}=\mathcal{O}(-k)\,\,\text{with\,\,}k:=7w(G)-5w(p_{3})-9w(u)-4w(v).
\]

\noindent$(2)$ The same statements hold as in $(1)$ for $\Pi_{\mathbb{P}}^{13}$
and for the ideal of $\Pi_{\mathbb{A}}^{14}$ in the polynomial ring
obtained from $S_{\Pi}$ by setting $L_{246}=1$. 
\end{prop}

By Proposition \ref{prop:916} and the proof of \cite[Cor.~4.2]{key-5},
we obtain the following:
\begin{cor}
\label{cor:QFanoData}For any class of a prime $\mathbb{Q}$-Fano
threefold as in Theorem $\ref{thm:main}$, we choose the weights of
coordinates as in there and a weighted complete intersection $X$
of $\Pi_{\mathbb{P}}^{13}$ or $\Pi_{\mathbb{P}}^{14}$ as in Table
$\ref{Table1}$. Then $K_{X}=\mathcal{O}_{X}(-1)$, and the genus
of $X$ and the Hilbert numerator of $X$ with respect to $-K_{X}$
coincide with those given in \cite{key-6}. 
\end{cor}

Moreover, by the proof of \cite[Cor.~4.3]{key-5}, we have
\begin{cor}
\label{cor:acX4}Let $X$ be as in Corollary $\ref{cor:QFanoData}$.
Assume that $X$ is quasi-smooth, has Picard number one, and has cyclic
quotient singularities as assigned for the class of $X$. Then ${\rm ac}_{X}=4$. 
\end{cor}

\section{Proof of Theorem \ref{thm:main}\label{sec:Proof-of-Theorem}}

In this section, we show Theorem \ref{thm:main}. We follow \cite[the subsec.~5.1]{key-5}
for the strategy how to verify a weighted complete intersection of
$\Pi_{\mathbb{P}}^{13}$ or $\Pi_{\mathbb{P}}^{14}$ in each case
is a desired example of a $\mathbb{Q}$-Fano threefold or a $K3$
suirface. 

\subsection{General strategy\label{subsec:General-strategy}}

\noindent\textbf{List of notation:}

We use the following notation for a weighted projective variety $V$
considered below (for example, $V=X,T$).
\begin{itemize}
\item $\mathbb{P}_{V}:=$the ambient weighted projective space for $V$.
\item $V_{\mathbb{A}}:=$the affine cone of $V$.
\item $V_{\mathbb{A}}^{o}:=$the complement of the origin in $V_{\mathbb{A}}$.
\item $\mathbb{A}_{V}:=$the ambient affine space for $V_{\mathbb{A}}$.
\item $\mathbb{A}_{V}^{*}:=$the affine space obtained from $\mathbb{A}_{V}$
by setting the coordinate $*=1$.
\item $V_{\mathbb{A}}^{*}:=$the restriction of $V_{\mathbb{A}}$ to $\mathbb{A}_{V}^{*}$.
\item $|\mathcal{O}_{\mathbb{P}_{V}}(i)|_{\mathbb{A}}:=$ the set of the
affine cones of members of $|\mathcal{\mathcal{O}}_{\mathbb{P}_{V}}(i)|$.
For this notation, we omit $\mathbb{P}_{V}$ if no confusion likely
occurs. 
\item ${\rm Bs}\,|\mathcal{O}_{\mathbb{P}_{V}}(i)|_{\mathbb{A}}:=$ the
intersection of all the members of $|\mathcal{\mathcal{O}}_{\mathbb{P}_{V}}(i)|_{\mathbb{A}}$.
\end{itemize}
\noindent\textbf{Check points (A), (B), (C):}

Recall the we set $b:=p_{3}^{2}+t_{1}p_{4}^{2}+p_{4}v$.
\begin{description}
\item [{(A)}] $X$ and $T$ are quasi-smooth, namely, $X_{\mathbb{A}}^{o}$
and $T_{\mathbb{A}}^{o}$ are smooth.
\item [{(B)}] $X$ and $T$ have only cyclic quotient singularities assigned
for No.$*$.
\item [{(C)}] $X\cap\{b=0\}$ is a prime divisor.
\end{description}
By these claims, we can finish the proof of Theorem \ref{thm:main}.
Indeed, by (A), (C) and Proposition \ref{prop:Pic1}, $X$ has Picard
number one. By this fact, (A), (B), and Corollaries \ref{cor:QFanoData}
and \ref{cor:acX4}, $X$ is an example of a $\mathbb{Q}$-Fano threefold
of No.$*$. Thus Theorem \ref{thm:main} (1-1) is proved. By (A) and
(B), Theorem \ref{thm:main} (2) is proved. 

For Claims (A) and (B), we have the following claim by \cite[Claim 5.1]{key-5}:
\begin{claim}
As for Claims (A) and (B), it suffices to show them for $T$.
\end{claim}

\vspace{3pt}

For Claim (C), we have only to show that $\dim\,\big(T\cap{\rm Sing}\,(X\cap\{b=0\})\big)\leq0$
by \cite[the last part of the subsec.~5.1]{key-5}. Since $T$ is
a Cartier divisor on $X$ outside a finite set of points, and $T\cap{\rm Sing}\,(X\cap\{b=0\})\subset{\rm Sing}\,(T\cap\{b=0\})$
where $T$ is a Cartier divisor, we have the following claim:
\begin{claim}
\label{claim:Cl(C)}As for Claim (C), it suffices to show that $\dim{\rm Sing}\,(T\cap\{b=0\})\leq0$.
\end{claim}

\subsection{Proof of Claims (A), (B), (C)\label{subsec:Proof-of-ClaimsABC}}

Note that, in the case that the key variety of $T$ is the weighted
cone $\Pi_{\mathbb{P}}^{14}$ over $\Pi_{\mathbb{P}}^{13}$ (No.~308,
501, 512, 550), the weight one coordinate is unique. Therefore $T$
is actually a weighted complete intersection of $\Pi_{\mathbb{P}}^{13}$.

\vspace{3pt}

\noindent\textbf{\fbox{No.308}} In this case, it holds that $T=\Pi_{\mathbb{P}}^{13}\cap(2)^{2}\cap(3)^{2}\cap(4)^{3}\cap(5)\cap(6)^{2}\cap(8)$
and then we can write $T_{\mathbb{A}}=\Pi_{\mathbb{A}}^{14}\cap\{t_{1}=L_{245}=t_{2}=L_{126}=p_{2}=L_{124}=L_{136}=0,\,L_{125}=a_{125}p_{1},\,L_{123}=a_{123}p_{4}+b_{123}u,\,L_{135}=a_{135}p_{4}+b_{135}u,\,s_{1}=a_{1}v\}\subset\mathbb{A}(p_{1},p_{4},u,p_{3},v,s_{2},s_{3})$,
where the parameters $a_{125},\dots,a_{1}\in\mathbb{C}$ are chosen
generally.

\vspace{3pt}

\noindent\textbf{Claim (A).} We show that $T_{\mathbb{A}}^{o}$
is smooth.

We show that $T_{\mathbb{A}}^{o}$ is smooth on the $v$-chart. We
set $T_{\mathbb{A}}(6):=\Pi_{\mathbb{A}}^{14}\cap\{t_{1}=L_{245}=t_{2}=L_{126}=p_{2}=L_{124}=L_{136}=0,\,L_{125}=a_{125}p_{1},\,L_{123}=a_{123}p_{4}+b_{123}u,\,L_{135}=a_{135}p_{4}+b_{135}u\}$,
which is obtained from $T_{\mathbb{A}}$ by removing the weight 8
equation $s_{1}=a_{1}v$. Since $w(v)=8$, the $v$-chart of $T_{\mathbb{A}}(6)$
is disjoint from ${\rm Bs}\,|\mathcal{O}_{\mathbb{P}_{T_{\mathbb{A}}(6)}}(8)|_{\mathbb{A}}$.
Therefore it suffices to check $T_{\mathbb{A}}^{o}(6)$ is smooth
on the $v$-chart by the Bertini theorem \cite{key-17}. We denote
by the same symbols the polynomials $F_{2},F_{4},F_{6}$ restricted
to $\mathbb{A}(p_{1},p_{4},u,p_{3},v,s_{2},s_{3})$ (we also use this
convention below). We see that, on the $v$-chart of $T_{\mathbb{A}}^{o}(6)$,
we may solve regularly the equation $F_{2}=F_{4}=F_{6}=0$ with respect
to $s_{1},s_{2},s_{3}$, namely, $s_{1},s_{2},s_{3}$ can be expressed
as polynomials of other variables. Replacing $s_{1},s_{2},s_{3}$
with these expressions, the equations of $T_{\mathbb{A}}^{o}(6)$
on the $v$-chart are reduced to one equation $-1+a_{123}p_{1}^{2}p_{4}-a_{125}p_{1}^{2}u+b_{123}p_{1}^{2}u-a_{123}p_{1}p_{3}p_{4}u+a_{125}p_{1}p_{3}u^{2}-b_{123}p_{1}p_{3}u^{2}-a_{123}p_{4}^{2}u^{2}-a_{135}p_{4}u^{3}-b_{123}p_{4}u^{3}-b_{135}u^{4}=0$.
By the Jacobian criterion, we see that this is smooth. Therefore,
$T_{\mathbb{A}}^{o}(6)$ is smooth on the $v$-chart and so is $T_{\mathbb{A}}^{o}$
as we noted above. 

We show that $T_{\mathbb{A}}^{o}$ is smooth on the $p_{1}$-chart.
Since we have checked that $T_{\mathbb{A}}^{o}$ is smooth on the
$v$-chart, we have only to consider this problem on the locus $\{v=0\}$.
We see that, on the $p_{1}$-chart of $T_{\mathbb{A}}^{o}$ near the
locus $\{v=0\}$, we may solve regularly $F_{1}=0,F_{3}=0,F_{5}=0$
with respect to $s_{2},s_{3},p_{4}$. Replacing $s_{2},s_{3},p_{4}$
with these expressions, the equations of $T_{\mathbb{A}}^{o}$ on
the $p_{1}$-chart near the locus $\{v=0\}$ are reduced to one equation,
of which we omit the explicit form since it is lengthy. By the Jacobian
criterion, we see that this is smooth. Therefore, $T_{\mathbb{A}}^{o}$
is smooth on the $p_{1}$-chart. 

It remains to check that $T_{\mathbb{A}}^{o}$ is smooth along the
locus $\{v=p_{1}=0\}$. It is easy to derive that $T\cap\{v=p_{1}=0\}$
is the union of the loci $S_{1}=\{a_{123}p_{4}^{2}+a_{135}p_{4}u+b_{123}p_{4}u+b_{135}u^{2}=0\}\subset\mathbb{A}(p_{4},u)$
and $S_{2}=\{a_{123}p_{4}^{3}-s_{2}^{2}=0\}\subset\mathbb{A}(p_{4},s_{2})$.
We see that $T_{\mathbb{A}}^{o}$ is smooth along $S_{1}$ and $S_{2}$
computing the Zariski tangent spaces of $T_{\mathbb{A}}^{o}$ at points
there. 

Thus we have shown that $T_{\mathbb{A}}^{o}$ is smooth.

\vspace{3pt}

\noindent\textbf{Claim (B).} We check the singularities of $T$
on each chart. The way to investigate is similar in all the charts,
we write the investigation in detail only for the $s_{3}$-chart of
$T$. 

For the $s_{3}$-chart of $T$, we consider $T_{\mathbb{A}}^{s_{3}}\subset\mathbb{A}_{T}^{s_{3}}$.
Since $w(s_{3})=10$, the cyclic group $\mathbb{Z}/10\mathbb{Z}$
acts on $\mathbb{A}_{T}^{s_{3}}$. It is easy to see that the nonfree
locus of the $\mathbb{Z}/10\mathbb{Z}$-action is contained in $\{p_{3}=s_{2}=0\}\subset\mathbb{A}_{T}^{s_{3}}$
and it holds that $T_{\mathbb{A}}^{s_{3}}\cap\{p_{3}=s_{2}=0\}=\{p_{3}=s_{2}=p_{1}=0,\,u=v^{2},\,p_{4}=a_{1}v^{2},\,1+(a_{1}^{2}a_{123}+a_{1}a_{135}+a_{1}b_{123}+b_{135})v^{5}=0\}$,
which consists of five points. By the $\mathbb{Z}/10\mathbb{Z}$-action,
these five points are mapped to one point on $T$, which we denote
by $\mathsf{p_{2}}$. The stabilizer groups at these five points are
isomorphic to $\mathbb{Z}/2\mathbb{Z}$. Since these five points are
isolated as the nonfree locus of the $\mathbb{Z}/2\mathbb{Z}$-action,
$T$ has a $1/2(1,1)$-singularity at $\mathsf{p}_{2}$ (cf. \cite[the subsec.5.1,  the last part of the item Claims (A) and (B) for T]{key-5}.

For the $p_{1}$-chart of $T$, we consider $T_{\mathbb{A}}^{p_{1}}\subset\mathbb{A}_{T}^{p_{1}}$.
Since $w(p_{1})=5$, the cyclic group $\mathbb{Z}/5\mathbb{Z}$ acts
on $\mathbb{A}_{T}^{p_{1}}$. It is easy to see that the nonfree locus
of the $\mathbb{Z}/5\mathbb{Z}$-action consist of the $p_{1}$-point.
Using the equations of $T,$ we can show that the tangent space of
$T_{\mathbb{A}}^{p_{1}}$ in $\mathbb{A}_{T}^{p_{1}}$ at the $p_{1}$-point
is $\{s_{2}=s_{3}=p_{4}=u=0\}.$ Therefore we may take $p_{3},v$
as a local coordinate of $T_{\mathbb{A}}^{p_{1}}$ at the $p_{1}$-point.
Since $w(p_{3})=7$ and $w(v)=8$, $T$ has only a $1/5(2,3)$-singularity
at the $p_{1}$-point. 

In a similar way, we see that, on the $s_{2}$-chart, $T$ have only
a $1/3(2,1)$-singularity at the point which is the image of the locus
$S_{2}$ defined in the proof of Claim (A) above. On the $v$-chart,
we see that the singularities of $T$ appear on the $s_{3}$-chart,
thus we have already checked them. We see that $T$ is nonsingular
on the $p_{3}$-chart. On the union of the $p_{4}$- and $u$-charts,
we see that $T$ has only two $1/6(1,5)$-singularities at the point
which is the image of the locus $S_{1}$ defined in the proof of Claim
(A) above. 

\vspace{3pt}

\noindent\textbf{Claim (C).} Using the description of the $v$-chart
of $T_{\mathbb{A}}(6)$ as in the proof of Claim (A) above, we can
easily check that $\dim{\rm Sing}\,(T\cap\{b=0\})\leq0$ on the $v$-chart
of $T$. Moreover, we can check that $(T\cap\{b=0\})\cap\{v=0\}$
consists of a finite number of points. Therefore, by Claim \ref{claim:Cl(C)},
Claim (C) follows for No.308.

\vspace{3pt}

In all the cases, the proofs of Claim (A)--(C) are similar, so we
only give outlines of the proofs in the cases below. 

\vspace{3pt}

\noindent\textbf{\fbox{No.501}} In this case, it holds that $T=\Pi_{\mathbb{P}}^{13}\cap(2)^{2}\cap(3)\cap(4)^{3}\cap(5)^{2}\cap(6)^{3}$
and then we can write $T_{\mathbb{A}}=\Pi_{\mathbb{A}}^{14}\cap\{t_{1}=0,\,L_{245}=0,\,t_{2}=aL_{126},\,p_{2}=0,L_{124}=0,L_{136}=0,\,p_{1}=0,\,L_{125}=0,p_{4}=b_{4}u+c_{4}L_{126}^{2},\,L_{123}=b_{123}u+c_{123}L_{126}^{2},\,L_{135}=b_{135}u+c_{135}L_{126}^{2}\}\subset\mathbb{A}(L_{126},u,p_{3},v,s_{1},s_{2},s_{3})$,
where the parameters $a,\dots,c_{135}\in\mathbb{C}$ are chosen generally.

\vspace{3pt}

\noindent\textbf{Claim (A).} We show that $T_{\mathbb{A}}^{o}$
is smooth.

We show that $T_{\mathbb{A}}^{o}$ is smooth on the $u$-chart. We
set $T_{\mathbb{A}}(5):=\Pi_{\mathbb{A}}^{15}\cap\{t_{1}=0,\,L_{245}=0,\,t_{2}=aL_{126},\,p_{2}=0,L_{124}=0,L_{136}=0,\,p_{1}=0,\,L_{125}=0\}$,
which is obtained from $T_{\mathbb{A}}$ by removing the weight 6
equations $p_{4}=b_{4}u+c_{4}L_{126}^{2},\,L_{123}=b_{123}u+c_{123}L_{126}^{2},\,L_{135}=b_{135}u+c_{135}L_{126}^{2}$
. Since $w(u)=6$, it suffices to check $T_{\mathbb{A}}^{o}(5)$ is
smooth on the $u$-chart. Similarly to the case of the $v$-chart
of $T_{\mathbb{A}}(6)$ of No.308, we can show that the $u$-chart
of $T_{\mathbb{A}}(5)$ is isomorphic to a hypersurface and this is
smooth.

We show that $T_{\mathbb{A}}^{o}$ is smooth on the $L_{126}$-chart.
Since we have already check the smoothness of $T_{\mathbb{A}}^{o}$
on the $u$-chart, we may assume that $u=0$. The condition that $L_{126}\not=0$
and $u=0$ on $T_{\mathbb{A}}$ induces the condition that $p_{4}\not=0$
on $T_{\mathbb{A}}(5)$. Since $w(p_{4})=6$, it suffices to check
$T_{\mathbb{A}}^{o}(5)$ is smooth on the $p_{4}$-chart. This can
be done also similarly to the case of the $v$-chart of $T_{\mathbb{A}}(6)$
of No.308.

Now we check that $T_{\mathbb{A}}^{o}$ is smooth along the locus
$\{u=L_{126}=0\}$. It is easy to derive that $T\cap\{u=L_{126}=0\}=\mathbb{A}(s_{1})$.
We can show that the Zariski tangent space in $\mathbb{A}_{T}$ at
a point of $\mathbb{A}(s_{1})$ is defined by $\{u=v=s_{3}=L_{126}=0\}$.

Thus $T_{\mathbb{A}}^{o}$ is smooth.

\vspace{3pt}

\noindent\textbf{Claim (B).} We check the singularities of $T$
on each chart.

On the $s_{3}$-chart, we can show similarly to the case of the $s_{3}$-chart
for No.308 that $T$ has a $1/2(1,1)$-singularity at $\{s_{3}=1,p_{3}=s_{2}=L_{126}=0,u=v^{2},s_{1}=b_{4}v,1+(b_{135}+b_{123}b_{4})v^{5}=0\}$
which consists of one point. For the $L_{126}$-chart, we consider
$T_{\mathbb{A}}^{L_{126}}\subset\mathbb{A}_{T}^{L_{126}}$. Since
$w(L_{126})=3$, the cyclic group $\mathbb{Z}/3\mathbb{Z}$ acts on
$\mathbb{A}_{T}^{L_{126}}$. It is easy to see that the nonfree locus
of the $\mathbb{Z}/3\mathbb{Z}$-action is the locus $\{u\not=0,s_{2}=(ac_{4}^{2}+2ab_{4}c_{4}u+ab_{4}^{2}u^{2})/u$,
$a^{2}c_{4}^{3}+(ac_{4}^{2}+3a^{2}b_{4}c_{4}^{2})u+(2ab_{4}c_{4}+3a^{2}b_{4}^{2}c_{4}-c_{123}c_{4})u^{2}+(ab_{4}^{2}+a^{2}b_{4}^{3}-b_{4}c_{123}-c_{135}-b_{123}c_{4})u^{3}-(b_{135}+b_{123}b_{4})u^{4}=0\}\subset\mathbb{A}(u,s_{2}),$
whose image on $T$ consists of $4$ points. By generality of $T$,
we may assume that $s_{2}\not=0$ for these four points. We show that
$T$ has $1/3(1,2)$-singularities at these 4 points. To see this,
we switch to the $L_{126}$-chart of $T_{\mathbb{A}}(5)$, which is
introduced in the proof of Claim (A) above. Let us denote by the $T(5)$
the weighted projectivization of $T_{\mathbb{A}}(5)$. We may assume
that $p_{4}s_{2}\not=0$ since $T$ is obtained from $T(5)$ by cutting
generally with three weight six hypersurfaces. It is easy to see that
the nonfree locus of the $\mathbb{Z}/3\mathbb{Z}$-action on $\{p_{4}s_{2}\not=0\}\subset\mathbb{A}_{T(5)}^{L_{126}}$
is $\{u=(ap_{4}^{2})/s_{2},\,L_{123}=(-aL_{135}p_{4}^{3}+p_{4}s_{2}^{2}+s_{2}^{3})/(p_{4}^{2}s_{2})\}\subset\mathbb{A}(p_{4},u,L_{123},L_{135},s_{2})$.
Then calculating the tangent spaces at points of this locus, we see
that $T(5)$ has $1/3(1,2)$-singularities along this locus. Therefore
$T$ has four $1/3(1,2)$-singularities on the $L_{126}$-chart. For
the $s_{1}$-chart, we see that $T$ has a $1/8(7,1)$-singularity
at the $s_{1}$-point. As for the $s_{2}$-, $v$-, and $u$-charts,
we see that $T$ has no other singularities. 

\vspace{3pt}

\noindent\textbf{Claim (C).} Using the description of the $u$-
and $p_{4}$-charts of $T_{\mathbb{A}}(5)$ obtained in the proof
of Claim (A), we can easily check that $\dim{\rm Sing}\,(T\cap\{b=0\})\leq0$
on the $u$- and $L_{126}$-charts of $T$. Moreover, we can check
that $(T\cap\{b=0\})\cap\{u=0,L_{126}=0\}$ consists of a finite number
of points. Therefore, by Claim \ref{claim:Cl(C)}, Claim (C) follows
for No.501.

\vspace{3pt}

\noindent\textbf{\fbox{No.512}} In this case, it holds that $T=\Pi_{\mathbb{P}}^{13}\cap(2)^{2}\cap(3)^{2}\cap(4)^{3}\cap(5)^{2}\cap(6)^{2}$
and then we can write $T_{\mathbb{A}}=\Pi_{\mathbb{A}}^{14}\cap\{t_{1}=0,\,L_{245}=0,\,p_{2}=a_{2}L_{126},\,t_{2}=aL_{126},\,p_{1}=0,\,L_{124}=0,\,L_{136}=0,\,p_{4}=a_{4}u,\,L_{125}=a_{125}u,\,L_{123}=a_{123}p_{3}+b_{123}L_{126}^{2},\,L_{135}=a_{135}p_{3}+b_{135}L_{126}^{2}\}\subset\mathbb{A}(L_{126},u,p_{3},v,s_{1},s_{2},s_{3})$,
where the parameters $a_{2},\dots,b_{135}\in\mathbb{C}$ are chosen
generally.

\vspace{3pt}

\noindent\textbf{Claim (A).} We show that $T_{\mathbb{A}}^{o}$
is smooth. We can show that $T_{\mathbb{A}}^{o}$ is smooth on the
$L_{126}$- and $u$-charts as in the other cases. As for the locus
$\{L_{126}=u=0\}$, it is easy to derive that $T\cap\{L_{126}=u=0\}$
is the locus $\mathbb{A}(s_{1})$. We see that $T_{\mathbb{A}}^{o}$
is smooth along $\mathbb{A}(s_{1})$ computing the Zariski tangent
spaces of $T_{\mathbb{A}}^{o}$ at points there. Thus $T_{\mathbb{A}}^{o}$
is smooth.

\vspace{3pt}

\noindent\textbf{Claim (B).} We check the singularities of $T$
on each chart. We see that $T$ has a $1/7(1,6)$-singularity at the
$s_{1}$-point on the $s_{1}$-chart, a $1/5(2,3)$-singularity at
the $u$-point on the $u$-chart, and five $1/3(1,2)$-singularities
on the $L_{126}$-chart. We also see that there are no other singularities.

\vspace{3pt}

\noindent\textbf{Claim (C).} We can check that $\dim{\rm Sing}\,(T\cap\{b=0\})\leq0$
on the $L_{126}$-chart of $T$ as in the other cases. Moreover, we
can check that $(T\cap\{b=0\})\cap\{L_{126}=0\}$ consists of a finite
number of points. Therefore, by Claim \ref{claim:Cl(C)}, Claim (C)
follows for No.512.

\vspace{3pt}

\noindent\textbf{\fbox{No.550}} In this case, it holds that $T=\Pi_{\mathbb{P}}^{13}\cap(2)^{3}\cap(3)^{2}\cap(4)^{3}\cap(5)\cap(6)^{2}$
and then we can write $T_{\mathbb{A}}=\Pi_{\mathbb{A}}^{14}\cap\{p_{2}=0,\,t_{1}=0,\,L_{245}=0,\,p_{1}=a_{1}L_{126},\,t_{2}=aL_{126},\,p_{4}=a_{4}u,\,L_{124}=a_{124}u,\,L_{136}=a_{136}u,\,L_{125}=a_{125}p_{3},\,L_{123}=a_{123}v+b_{123}s_{1}+c_{123}L_{126}^{2},\,L_{135}=a_{135}v+b_{135}s_{1}+c_{135}L_{126}^{2}\}\subset\mathbb{A}(L_{126},u,p_{3},v,s_{1},s_{2},s_{3})$,
where the parameters $a_{1},\dots,c_{135}\in\mathbb{C}$ are chosen
generally.

\vspace{3pt}

\noindent\textbf{Claim (A).} We show that $T_{\mathbb{A}}^{o}$
is smooth. We can show that $T_{\mathbb{A}}^{o}$ is smooth on the
$L_{126}$- and $p_{3}$-charts as in the other cases. As for the
locus $\{L_{126}=p_{3}=0\}$, it is easy to derive that $T\cap\{L_{126}=p_{3}=0\}$
is the union of the three loci $\mathbb{A}(s_{1})$, $S_{1}:=\{u\not=0,L_{126}=p_{3}=s_{1}=s_{2}=v=0,s_{3}=1/3a_{136}u^{2}\}$,
and $S_{2}:=\{u\not=0,L_{126}=p_{3}=s_{2}=0,s_{1}=a_{4}v,s_{3}=1/3(-3a_{135}-2a_{136}-3a_{123}a_{4}-3a_{124}a_{4}-3a_{4}^{2}b_{123}-3a_{4}b_{135})u^{2},(a_{135}+a_{136}+a_{123}a_{4}+a_{124}a_{4}+a_{4}^{2}b_{123}+a_{4}b_{135})u^{3}+v^{2}=0\}$.
We see that $T_{\mathbb{A}}^{o}$ is smooth along $S_{1}$ and $S_{2}$
computing the Zariski tangent spaces of $T_{\mathbb{A}}^{o}$ at points
in them. Thus $T_{\mathbb{A}}^{o}$ is smooth.

\vspace{3pt}

\noindent\textbf{Claim (B).} We check the singularities of $T$
on each chart. On the $L_{126}$-chart, we can show that $T$ has
three $1/3(1,2)$-singularities similarly to the case of the $L_{126}$-chart
of No.501. On the $u$-chart, we see that $T$ has a $1/4$(1,3)-singularity
at the image of the locus $S_{1}$, and has a $1/2(1,1)$-singularity
at the image of the locus $S_{2}$. On the $s_{1}$-chart, $T$ has
a $1/6(1,5)$-singularity at the $s_{1}$-point. We see that $T$
has no other singularities.

\vspace{3pt}

\noindent\textbf{Claim (C).} We can easily check that $\dim{\rm Sing}\,(T\cap\{b=0\})\leq0$
on the $L_{126}$-chart of $T$ as in the other cases. Moreover, we
can check that $(T\cap\{b=0\})\cap\{L_{126}=0\}$ consists of a finite
number of points. Therefore, by Claim \ref{claim:Cl(C)}, Claim (C)
follows for No.550.

\vspace{3pt}

\noindent\textbf{\fbox{No.872} }Unlike the other cases, we first
take a coordinate change $t_{2}=aT_{2}+bl_{126},\,L_{126}=cT_{2}+dl_{126}$
for $\Pi_{\mathbb{A}}^{15}$, where $a,\,b,\,c,\,d\in\mathbb{C}$
are general constants and $T_{2},l_{126}$ are new coordinates replacing
$t_{2},\,L_{126}$. Then, since $T=\Pi_{\mathbb{P}}^{14}\cap(1)\cap(2)^{3}\cap(3)^{2}\cap(4)^{2}\cap(5)\cap(6)^{2}$,
we can write $T_{\mathbb{A}}=\Pi_{\mathbb{A}}^{15}\cap\{p_{1}=0,\,t_{1}=0,\,L_{245}=0,\,p_{4}=T_{2},\,u=l_{126},\,L_{124}=a_{124}p_{3},L_{136}=a_{136}p_{3},\,L_{125}=a_{125}v+b_{125}s_{1},\,L_{123}=a_{123}s_{2}+q_{1}(T_{2},l_{126}),\,L_{135}=a_{135}s_{2}+q_{2}(T_{2},l_{126})\}\subset\mathbb{A}(T_{2},l_{126},p_{3},v,s_{1},s_{2},s_{3})$,
where the parameters $a_{124},\dots,a_{135}\in\mathbb{C}$ and quadratic
forms $q_{1}(T_{2},l_{126}),q_{2}(T_{2},l_{126})$ are chosen generally.

\vspace{3pt}

\noindent\textbf{Claim (A).} We can show that $T_{\mathbb{A}}^{o}$
is smooth on the $T_{2}$- and the $l_{126}$-charts as in the other
cases. It remains to check that $T_{\mathbb{A}}^{o}$ is smooth along
the locus $\{T_{2}=l_{126}=0\}$. It is easy to derive that $T\cap\{T_{2}=l_{126}=0\}=\mathbb{A}(s_{1})$.
We see that $T_{\mathbb{A}}^{o}$ is smooth along $\mathbb{A}(s_{1})$
computing the Zariski tangent spaces of $T_{\mathbb{A}}^{o}$ at points
in them. Thus $T_{\mathbb{A}}^{o}$ is smooth.

\vspace{3pt}

\noindent\textbf{Claim (B).} We check the singularities of $T$
on each chart. We can show that $T$ has a $1/5(1,4)$-singularity
at the $s_{1}$-point on the $s_{1}$-chart. On the $T_{2}$-chart,
we can show that $T$ has five $1/3(1,2)$-singularities similarly
to the case of the $L_{126}$-chart of No.501. We see that $T$ has
no other singularities.

\vspace{3pt}

\noindent\textbf{Claim (C).} Using the descriptions of the $T_{2}$-
and $l_{126}$-charts of $T_{\mathbb{A}}$, we can easily check that
$\dim{\rm Sing}\,(T\cap\{b=0\})\leq0$ on the $T_{2}$- and $l_{126}$-charts
of $T$. Moreover, we can check that $(T\cap\{b=0\})\cap\{T_{2}=l_{126}=0\}$
consists of a finite number of points. Therefore, by Claim \ref{claim:Cl(C)},
Claim (C) follows for No.872.

\vspace{1cm}

\noindent\textbf{\fbox{No.577}} Since $T=\Pi_{\mathbb{P}}^{13}\cap(1)\cap(2)^{3}\cap(3)^{4}\cap(4)^{2}\cap(5)$,
we can write $T_{\mathbb{A}}=\Pi_{\mathbb{A}}^{14}\cap\{L_{245}=t_{1}=L_{126}=L_{136}=0,\,p_{2}=a_{2}L_{135},\,t_{2}=aL_{135},\,L_{124}=a_{124}L_{135},\,L_{125}=a_{125}L_{135},\,p_{1}=a_{1}L_{123},\,p_{4}=a_{4}L_{123},\,p_{3}=a_{3}u+b_{3}s_{1}\}$
$\subset\mathbb{A}(L_{135},L_{123},u,s_{1},v,s_{2},s_{3})$.

\vspace{3pt}

\noindent\textbf{Claim (A).} We show that $T_{\mathbb{A}}^{o}$
is smooth. The arguments here are slightly involved.

We show that $T_{\mathbb{A}}^{o}$ is smooth on the $L_{123}$-chart.
By the equation of $T_{\mathbb{A}}$ as above, this is contained in
$\Pi_{\mathbb{A}}^{15}\cap\{p_{1}p_{4}L_{123}\not=0\}$. We set $T_{\mathbb{A}}'=\Pi_{\mathbb{A}}^{14}\cap\{L_{245}=t_{1}=L_{126}=L_{136}=0,\,p_{2}=a_{2}L_{135},\,t_{2}=aL_{135},\,L_{124}=a_{124}L_{135},\,L_{125}=a_{125}L_{135},\,p_{3}=a_{3}u+b_{3}s_{1}\}$,
which is obtained from $T_{\mathbb{A}}$ by removing the weight 4
equations $p_{1}=a_{1}L_{123},\,p_{4}=a_{4}L_{123}$. As we noted
above, the $L_{123}$-chart of $T_{\mathbb{A}}$ is contained in the
$p_{1}$-chart of $T_{\mathbb{A}}'$. Since $w(p_{1})=4$, the $p_{1}$-chart
of $T_{\mathbb{A}}'$ is disjoint from ${\rm Bs}\,|\mathcal{O}(4)|_{\mathbb{A}}$.
Therefore it suffices to check that ${\rm Sing}\,T_{\mathbb{A}}'$
is of codimension $\geq3$ in $T_{\mathbb{A}}'$ on the $p_{1}$-chart
assuming $p_{4}\not=0$ since then the singular locus will disappear
after cutting by the weight 4 hypersurfaces. If $L_{135}=0$ or $u=0$,
then we can show this by the Jacobian criterion since we see that
$T_{\mathbb{A}}'$ is isomorphic to a hypersurface on the $p_{1}$-chart
near $\{L_{135}=0\}$ or $\{u=0\}$. Thus we have only to consider
the problem on the locus $p_{1}p_{4}L_{135}u\not=0$. Note that this
is outside ${\rm Bs}\,|\mathcal{O}(3)|_{\mathbb{A}}\cup{\rm Bs}\,|\mathcal{O}(4)|_{\mathbb{A}}\cup{\rm Bs}\,|\mathcal{O}(5)|_{\mathbb{A}}$
and is contained in $\Pi_{\mathbb{A}}^{15}\cap\{p_{2}p_{4}\not=0\}$.
By the description of $\Pi_{\mathbb{A}}^{15}\cap\{p_{2}p_{4}\not=0\}$
as in the proof of Proposition \ref{prop:SingPi}, we see that $\Pi_{\mathbb{A}}^{14}\cap\{L_{245}=t_{1}=L_{126}=L_{136}=0\}$
has only $c(G(2,5))$-singularities since ${\rm Sing}\,(\Pi_{\mathbb{A}}^{15}\cap\{p_{2}p_{4}\not=0\})$
is the open subset $\{p_{2}p_{4}\not=0\}$ of the affine space $\mathbb{A}(p_{2},p_{3},p_{4},t_{1},L_{126},L_{136},L_{245},L_{246})$.
Since we consider outside ${\rm Bs}\,|\mathcal{O}(3)|_{\mathbb{A}}\cup{\rm Bs}\,|\mathcal{O}(4)|_{\mathbb{A}}\cup{\rm Bs}\,|\mathcal{O}(5)|_{\mathbb{A}}$,
these singularities of $\Pi_{\mathbb{A}}^{14}\cap\{L_{245}=t_{1}=L_{126}=L_{136}=0\}$
will disappear by cutting weight $3,\,4,\,5$ hypersurfaces. Thus
$T_{\mathbb{A}}^{o}$ is smooth on the $L_{123}$-chart.

We show that $T_{\mathbb{A}}^{o}$ is smooth on the $u$-chart. Set
$T_{\mathbb{A}}''=\Pi_{\mathbb{A}}^{14}\cap\{L_{245}=t_{1}=L_{126}=L_{136}=0,\,p_{2}=a_{2}L_{135},\,t_{2}=aL_{135},\,L_{124}=a_{124}L_{135},\,L_{125}=a_{125}L_{135},\,p_{1}=a_{1}L_{123},\,p_{4}=a_{4}L_{123}\}$,
which is obtained from $T_{\mathbb{A}}$ by removing the weight 5
equation $p_{3}=a_{3}u+b_{3}s_{1}$. Since we have already checked
the smoothness of $T_{\mathbb{A}}^{o}$ on the $L_{123}$-chart, we
may assume $L_{123}=0$. On the $u$-chart of $T_{\mathbb{A}}''$
near the locus $\{L_{123}=0\}$, we may solve regularly $F_{1}=0,F_{3}=0$
with respect to $s_{1},s_{2}$. Then we see that the $u$-chart of
$T_{\mathbb{A}}''$ is isomorphic to a complete intersection of codimension
two near the locus $\{L_{123}=0\}$ and has at most isolated singularities.
Since $w(u)=5$, the $u$- chart of $T_{\mathbb{A}}''$ is disjoint
from ${\rm Bs}\,|\mathcal{O}(5)|_{\mathbb{A}}$. Thus $T_{\mathbb{A}}^{o}$
is smooth on the $u$-chart.

We see that $T_{\mathbb{A}}$ is smooth along $T_{\mathbb{A}}\cap\{L_{123}=u=0\}$
computing the Zariski tangent spaces of $T_{\mathbb{A}}^{o}$ at points
in them.

Thus $T_{\mathbb{A}}^{o}$ is smooth.

\noindent\textbf{Claim (B).} We check the singularities of $T$
on each chart; $T$ has two $1/5(1,4)$-singularities on the $u$-chart,
three $1/3(1,2)$-singularities on the $L_{135}$-chart, and one $1/2(1,1)$-singularity
on the $L_{123}$-chart. We see that $T$ has no other singularities.

\vspace{3pt}

\noindent\textbf{Claim (C).} The arguments here are also slightly
involved. We can easily check that $(T\cap\{b=0\})\cap\{L_{135}L_{123}u=0\}$
consists of a finite number of points. Thus we have only to show that
$T_{\mathbb{A}}\cap\{b=0\}$ is smooth on $\{L_{135}L_{123}u\not=0\}$.
Note that $T\cap\{L_{135}L_{123}u\not=0\}$ is contained in $\Pi_{\mathbb{A}}^{14}\cap\{L_{245}=t_{1}=L_{126}=L_{136}=0\}\cap\{p_{2}p_{4}u\not=0\}$.
Then we can use the description of $\Pi_{\mathbb{A}}^{15}\cap\{p_{2}p_{4}\not=0\}$
as in the proof of Proposition \ref{prop:SingPi} and we see that
$\Pi_{\mathbb{A}}^{14}\cap\{L_{245}=t_{1}=L_{126}=L_{136}=0\}\cap\{b=0\}$
on $\{p_{2}p_{4}u\not=0\}$ is isomorphic to the cone over the open
subset $G(2,5)\cap\{m_{25}=0\}$ corresponding to $u\not=0$. Therefore
the codimension of the singular locus of $\Pi_{\mathbb{A}}^{14}\cap\{L_{245}=t_{1}=L_{126}=L_{136}=0\}\cap\{b=0\}$
on $\{p_{2}p_{4}u\not=0\}$ is $3$ since so is the codimension of
${\rm Sing}\,(G(2,5)\cap\{m_{25}=0\})$. Since $\Pi_{\mathbb{A}}^{14}\cap\{L_{245}=t_{1}=L_{126}=L_{136}=0\}\cap\{p_{2}p_{4}u\not=0\}$
is disjoint from ${\rm Bs}\,|\mathcal{O}(3)|_{\mathbb{A}}\cup{\rm Bs}\,|\mathcal{O}(4)|_{\mathbb{A}}\cup{\rm Bs}\,|\mathcal{O}(5)|_{\mathbb{A}}$,
the singular locus of $\Pi_{\mathbb{A}}^{14}\cap\{L_{245}=t_{1}=L_{126}=L_{136}=0\}\cap\{b=0\}$
will disappear by cutting weight $3,\,4,\,5$ hypersurfaces. Thus
$T_{\mathbb{A}}\cap\{b=0\}$ is smooth on $\{L_{135}L_{123}u\not=0\}$.
Therefore, by Claim \ref{claim:Cl(C)}, Claim (C) follows for No.577.

\vspace{3pt}

\noindent\textbf{\fbox{No.878}} Since $T=\Pi_{\mathbb{P}}^{13}\cap(1)\cap(2)^{4}\cap(3)^{4}\cap(4)^{2}$,
we can write $T_{\mathbb{A}}=\Pi_{\mathbb{A}}^{14}\cap\{L_{245}=p_{2}=t_{1}=L_{126}=L_{136}=0,p_{1}=a_{1}L_{125}+b_{1}L_{135},p_{4}=a_{4}L_{125}+b_{4}L_{135},t_{2}=aL_{125}+bL_{135},L_{124}=a_{124}L_{125}+b_{124}L_{135},p_{3}=a_{3}s_{1}+b_{3}L_{123},u=cs_{1}+dL_{123}\}\subset\mathbb{A}(L_{125,}L_{135},s_{1},L_{123},v,s_{2},s_{3})$,
where the parameters $a_{1},\dots,d\in\mathbb{C}$ are chosen generally.

\vspace{3pt}

\noindent\textbf{Claim (A).} We show that $T_{\mathbb{A}}^{o}$
is smooth.

We show that $T_{\mathbb{A}}^{o}$ is smooth if $a_{1}L_{125}+b_{1}L_{135}\not=0$
or $a_{4}L_{125}+b_{4}L_{135}\not=0$. Set $T_{\mathbb{A}}'=\Pi_{\mathbb{A}}^{14}\cap\{L_{245}=p_{2}=t_{1}=L_{126}=L_{136}=0,p_{3}=a_{3}s_{1}+b_{3}L_{123},u=cs_{1}+dL_{123}\}$,
which is obtained from $T_{\mathbb{A}}$ by removing the weight 3
equations $p_{1}=a_{1}L_{125}+b_{1}L_{135},p_{4}=a_{4}L_{125}+b_{4}L_{135},t_{2}=aL_{125}+bL_{135},L_{124}=a_{124}L_{125}+b_{124}L_{135}$
. Since $w(p_{1})=w(p_{4})=3$, the $p_{1}$- and $p_{4}$-charts
of $T_{\mathbb{A}}'$ is disjoint from ${\rm Bs}\,|\mathcal{O}_{\mathbb{P}_{T_{\mathbb{A}}'}}(3)|_{\mathbb{A}}$.
Therefore it suffices to check that ${\rm Sing}\,T_{\mathbb{A}}'$
is of codimension $\geq3$ in $T_{\mathbb{A}}'$ on the $p_{1}$-
and $p_{4}$-charts since then the singular locus will disappear after
cutting by the four weight 3 hypersurfaces. We can show this by the
Jacobian criterion since we see that $T_{\mathbb{A}}'$ is isomorphic
to hypersurfaces on the $p_{1}$- and $p_{4}$-charts . Therefore
$T_{\mathbb{A}}^{o}$ is smooth if $a_{1}L_{125}+b_{1}L_{135}\not=0$
or $a_{4}L_{125}+b_{4}L_{135}\not=0$. 

Note that $T_{\mathbb{A}}\cap\{a_{1}L_{125}+b_{1}L_{135}=a_{4}L_{125}+b_{4}L_{135}=0\}=T_{\mathbb{A}}\cap\{L_{125}=L_{135}=0\}$
since we cut $T_{\mathbb{A}}'$ generally. Then we see that $T_{\mathbb{A}}$
is smooth along $T_{\mathbb{A}}\cap\{L_{125}=L_{135}=0\}$ computing
the Zariski tangent spaces of $T_{\mathbb{A}}^{o}$ at points there.

\vspace{3pt}

\noindent\textbf{Claim (B).} We check the singularities of $T$
on each chart; $T$ has two $1/4(1,3)$-singularities on the $s_{1}$-chart
and four $1/3(1,2)$-singularities on the $L_{125}$-chart. We see
that $T$ has no other singularities.

\vspace{3pt}

\noindent\textbf{Claim (C).} Using the description of the $p_{4}$-chart
of $T_{\mathbb{A}}'$, we can easily check that $\dim{\rm Sing}\,(T\cap\{b=0\})\leq0$
on $T$ if $a_{4}L_{125}+b_{4}L_{135}\not=0$. Moreover, we can check
that $(T\cap\{b=0\})\cap\{a_{4}L_{125}+b_{4}L_{135}=0\}$ consists
of a finite number of points. Therefore, by Claim \ref{claim:Cl(C)},
Claim (C) follows for No.878.

\vspace{3pt}

\noindent\textbf{\fbox{No.1766}} Since $T=\Pi_{\mathbb{P}}^{13}\cap(1)^{2}\cap(2)^{4}\cap(3)^{4}\cap(4)$,
it is possible to write $T_{\mathbb{A}}=\Pi_{\mathbb{A}}^{14}\cap\{p_{2}=L_{245}=0,\,p_{1}=a_{1}L_{136},\,p_{4}=a_{4}L_{136},\,t_{1}=a_{0}L_{136},\,L_{126}=a_{126}L_{136},\,p_{3}=a_{3}u+b_{3}s_{1}+c_{3}L_{135},\,t_{2}=au+bs_{1}+cL_{135},\,L_{124}=a_{124}u+b_{124}s_{1}+c_{124}L_{135},\,L_{125}=a_{125}u+b_{125}s_{1}+c_{125}L_{135},\,s_{2}=a_{2}v+b_{2}L_{123}+c_{2}L_{136}^{2}\}\subset\mathbb{A}(L_{136},u,s_{1},L_{135},v,L_{123},s_{3})$,
where the parameters $a_{1},\dots,c_{2}\in\mathbb{C}$ are chosen
generally. 

\noindent\textbf{Claim (A).} We show that $T_{\mathbb{A}}^{o}$
is smooth. Unlike the other cases, we show this separately on the\textit{
three }charts and the complement of their union.

We show that $T_{\mathbb{A}}^{o}$ is smooth on the $L_{136}$-chart.
We set $T_{\mathbb{A}}'=\Pi_{\mathbb{A}}^{14}\cap\{p_{2}=L_{245}=0,\,p_{3}=a_{3}u+b_{3}s_{1}+c_{3}L_{135},\,t_{2}=au+bs_{1}+cL_{135},\,L_{124}=a_{124}u+b_{124}s_{1}+c_{124}L_{135},\,L_{125}=a_{125}u+b_{125}s_{1}+c_{125}L_{135}\}$,
which is obtained from $T_{\mathbb{A}}$ by removing the weight 2
equations $p_{1}=a_{1}L_{136},\,p_{4}=a_{4}L_{136},\,t_{1}=a_{0}L_{136},\,L_{126}=a_{126}L_{136}$
and the weight 4 equation $s_{2}=a_{2}v+b_{2}L_{123}+c_{2}L_{136}^{2}$.
Similarly to the case of the $p_{1}$- and the $p_{4}$-charts of
No.878, we can show that $T_{\mathbb{A}}^{o}$ is smooth on the $L_{136}$-chart
checking the singularities of the $p_{1}$-chart of $T_{\mathbb{A}}'$.

We show that $T_{\mathbb{A}}^{o}$ is smooth on the $u$- and $s_{1}$-charts.
For this, we set $T_{\mathbb{A}}''=\Pi_{\mathbb{A}}^{14}\cap\{p_{2}=L_{245}=0,\,p_{1}=a_{1}L_{136},\,p_{4}=a_{4}L_{136},\,t_{1}=a_{0}L_{136},\,L_{126}=a_{126}L_{136},\,s_{2}=a_{2}v+b_{2}L_{123}+c_{2}L_{136}^{2}\}$,
which is obtained from $T_{\mathbb{A}}$ by removing the weight 3
equations $p_{3}=a_{3}u+b_{3}s_{1}+c_{3}L_{135},\,t_{2}=au+bs_{1}+cL_{135},\,L_{124}=a_{124}u+b_{124}s_{1}+c_{124}L_{135},\,L_{125}=a_{125}u+b_{125}s_{1}+c_{125}L_{135}$.
Similarly to the above cases, we can show that $T_{\mathbb{A}}''$
is isomorphic to a smooth hypersurface on the $u$-chart, and then
$T_{\mathbb{A}}$ is also smooth on the $u$-chart. As for the $s_{1}$-chart
of $T_{\mathbb{A}}''$, we have only to show the smoothness assuming
$L_{136}=0$. we can show that $T_{\mathbb{A}}''$ is isomorphic to
a smooth hypersurface on the $s_{1}$-chart near $\{L_{136}=0\}$,
and then $T_{\mathbb{A}}$ is also smooth on the $s_{1}$-chart.

Finally, we see that $T_{\mathbb{A}}\cap\{L_{136}=u=s_{1}=0\}$ is
empty, thus we have shown that $T_{\mathbb{A}}^{o}$ is smooth. 

\vspace{3pt}

\noindent\textbf{Claim (B). }We check the singularities of $T$
on each chart; $T$ has five $1/3(1,2)$-singularities on the $u$-chart
and $T$ has two $1/2(1,1)$-singularities on the $L_{136}$-chart.
We see that $T$ has no other singularities.

\vspace{3pt}

\noindent\textbf{Claim (C).} Using the description of the $u$-chart
of $T_{\mathbb{A}}''$, we can easily check that $\dim{\rm Sing}\,(T\cap\{b=0\})\leq0$
on the $u$-chart of $T$. Moreover, we can check that $(T\cap\{b=0\})\cap\{u=0\}$
consists of a finite number of points. Therefore, by Claim \ref{claim:Cl(C)},
Claim (C) follows for No.1766.

\subsection{Projection of $X$ to a hypersurface}

In this subsection, we show Theorem \ref{thm:main} (1-2) and (1-3). 
\begin{prop}
\label{prop:The-projection-key} The intersection of $\Pi_{\mathbb{P}}^{13}$
or $\Pi_{\mathbb{P}}^{14}$ with $\mathbb{P}(s_{1},s_{2},s_{3})$
is nonreduced and consists of only the $s_{1}$-point. The projection
of $\Pi_{\mathbb{P}}^{13}$ or $\Pi_{\mathbb{P}}^{14}$ from the intersection
with $\mathbb{P}(s_{1},s_{2},s_{3})$ is birational to the hypersurface
$\{G=0\}|_{L_{246}=1}$ or the cone over the hypersurface $\{G=0\}|_{L_{246}=1}$
respectively. The exceptional locus of the projection is contained
in the intersection of $\Pi_{\mathbb{P}}^{13}$ or $\Pi_{\mathbb{P}}^{14}$
with $\{D_{ijk}=0\,(1\leq i<j<k\leq6)\}$, which is a proper closed
subset of $\Pi_{\mathbb{P}}^{13}$ or $\Pi_{\mathbb{P}}^{14}$ respectively.
\end{prop}

\begin{proof}
In the proof, we only consider the case of $\Pi_{\mathbb{P}}^{13}$
since the treatment of the case of $\Pi_{\mathbb{P}}^{14}$ is almost
identical. The first assertion follows from a direct calculation. 

Since we may check by \cite{key-18} that the elimination ideal $I_{\Pi}\cap R_{G}$
of $I_{\Pi}$ is generated by the polynomial $G$, we see that the
projection $\Pi_{\mathbb{P}}^{13}\dashrightarrow\{G=0\}|_{L_{246}=1}$
is induced as above and is dominant applying \cite[Chap.3, the sect.2, Thm.3]{key-19}
and taking the weighted projectivization. In particular, $G|_{L_{246}=1}$
is irreducible.

Now we use the construction as in the proof of Proposition \ref{prop:Unproj}.
Note that $D_{123}\not\equiv$0 on $\{G=0\}|_{L_{246}=1}$ since $G|_{L_{246}=1}$
is irreducible and the weight of $D_{123}$ is smaller than that of
$G|_{L_{246}=1}$. Therefore over a general point of $\{G=0\}|_{L_{246}=1}$,
$s_{1},s_{2},s_{3}$ are defined by (\ref{eq:s1s2s3}) with $(i,j,k)=(1,2,3)$
and then we may verify all the $F_{l}$ ($1\leq l\leq9)$ vanish.
Actually this is true if we choose other triplets $(i,j,k)$. This
means that the complement of $\{D_{ijk}=0\,(1\leq i<j<k\leq6)\}$
in $\{G=0\}|_{L_{246}=1}$ isomorphically lifts to $\Pi_{\mathbb{P}}^{13}$.
Thus the projection is birational. By this argument, we have shown
also the last assertion.
\end{proof}

\subsubsection{No. 501, 512 \label{subsec:No.-501,-512,}}

In these cases, we can check by direct computations that $X\cap\mathbb{P}(s_{1},s_{2},s_{3})$
is nonreduced and consists of only the $s_{1}$-point, which is the
unique highest index cyclic quotient singularity. Then, in the almost
same way as the proof of Proposition \ref{prop:The-projection-key},
we can show Theorem \ref{thm:main} (1-2). Actually, we also see that
the restriction of the elimination of $s_{1},s_{2},s_{3}$ for $\Pi_{\mathbb{P}}^{13}$
to $X$ coincides with the elimination of $s_{1},s_{2},s_{3}$ for
$X$.

\subsubsection{No. 550, 872\label{subsec:No.-550,-872}}

In these cases, $X$ has also the unique highest index cyclic quotient
singularity and this is the $s_{1}$-point. This, however, does not
coincide with the $s_{1}$-point of $\Pi_{\mathbb{P}}^{13}$ or $\Pi_{\mathbb{P}}^{14}$.
To obtain the desired projection of $X$, we will perform a homogeneous
coordinate change of $\Pi_{\mathbb{P}}^{13}$ or $\Pi_{\mathbb{P}}^{14}$.
After the coordinate change, the proof of Proposition \ref{prop:The-projection-key}
works for $X$ to show Theorem \ref{thm:main} (1-3) in these cases,
so we only write down the coordinate change of $\Pi_{\mathbb{P}}^{13}$
or $\Pi_{\mathbb{P}}^{14}$. After the coordinate change below, the
restriction of the elimination of $s_{1},s_{2},s_{3}$ for $\Pi_{\mathbb{P}}^{13}$
to $X$ coincides with the elimination of $s_{1},s_{2},s_{3}$ for
$X$.

\vspace{3pt}

\noindent \textbf{No.550:} By the equation given in the subsection
\ref{subsec:Proof-of-ClaimsABC}, the $s_{1}$-point of $X$ of No.550
is the point with $s_{1}=1,L_{125}=b_{125},L_{135}=b_{135}$ and the
other coordinates being zero. Thus, by taking the coordinate change
$L_{125}'=L_{125}-b_{125}s_{1}$, $L_{135}'=L_{135}-b_{135}s_{1}$,
where $L_{123}',L_{135}'$ are new coordinates, the $s_{1}$-point
of $X$ coincides with that of $\Pi_{\mathbb{P}}^{14}$.

\vspace{3pt}

\noindent \textbf{No.872:} By the equation given in the subsection
\ref{subsec:Proof-of-ClaimsABC}, the $s_{1}$-point of $X$ of No.872
is the point with $s_{1}=1,L_{125}=b_{125}$ and the other coordinates
being zero. Thus, by taking the coordinate change $L_{125}'=L_{125}-b_{125}s_{1}$,
$L_{123}'=L_{123}-b_{123}s_{2}$, $L_{135}'=L_{135}-b_{135}s_{2}$,
where $L_{123}',L_{125}',L_{135}'$ are new coordinates, the $s_{1}$-point
of $X$ coincides with that of $\Pi_{\mathbb{P}}^{14}$ (we note that
we introduce the new coordinates $L_{123}'$ and $L_{135}'$ to make
the restriction of the elimination of $s_{1},s_{2},s_{3}$ for $\Pi_{\mathbb{P}}^{13}$
to $X$ coincides with the elimination of $s_{1},s_{2},s_{3}$ for
$X$).

\subsubsection{No. 577, 878, 1766\label{subsec:No.-577,-878,}}

In these cases, $\Sigma_{\mathbb{P}}^{13}$ is the key variety for
$X$. We can check by direct computations that $X\cap\mathbb{P}(s_{1},s_{2},s_{3})$
is empty, thus we need to find different projections of $\Sigma_{\mathbb{P}}^{13}$
from those as above. To obtain the desired projection of $X$, we
will perform a homogeneous coordinate change of $\Pi_{\mathbb{P}}^{13}$.
After the coordinate change, the proof of Proposition \ref{prop:The-projection-key}
works for $X$ to show Theorem \ref{thm:main} (1-3) in these cases,
so we only write down the coordinate change of $\Pi_{\mathbb{P}}^{13}$.

\vspace{3pt}

\noindent \textbf{No.577:} In this case, we seek a projection of
$X$ from one of two $1/5(1,1,4)$-singularities. Note that the index
5 locus of $\Sigma_{\mathbb{P}}^{13}$ is $\{p_{3}^{2}=s_{1}u\}\subset\mathbb{P}(p_{3},u,s_{1})$.
By the $(\mathbb{C}^{*})^{4}$-action of $\Sigma_{\mathbb{P}}^{13}$
(Proposition \ref{prop:GL}), any general point of this locus can
be moved to the point $\mathsf{p=}\{p_{3}=u=s_{1}\}\in\mathbb{P}(p_{3},u,s_{1})$.
Thus we may assume that $X$ contain the point $\mathsf{p}.$ Then
$\mathsf{p}$ is one of two $1/5(1,1,4)$-singularities. We will see
that a projection with $\mathsf{p}$ being the center is the desired
one. 

For calculations, we perform the following coordinate change: $p'_{3}=p_{3}-s_{1},u'=-2p_{3}+s_{1}+u,v'=1/3(L_{245}s_{1})-s_{2}+v,s_{3}'=1/3(-L_{136}s_{1})+s_{3},t_{2}'=-L_{124}+L_{125}-L_{135}+t_{2}$
replacing $p_{3},u,v,s_{3},t_{2}$, and the other coordinates are
unchanged. Then the point $\mathsf{p}$ becomes the $s_{1}$-point
of $X$ and the tangent space of the local index 5 cover of $\Sigma_{\mathbb{P}}^{13}$
at the inverse image of $\mathsf{p}$ is defined by $\{u'=v'=s_{3}'=t_{2}'=0\}.$
The projection of $\Sigma_{\mathbb{P}}^{13}$ induced by the elimination
of $s_{1},s_{2},s_{3}'$ and its restriction to $X$ are the desired
projections.

\vspace{3pt}

\noindent \textbf{No.878:} In this case, we seek a projection of
$X$ from one of two $1/4(1,1,3)$-singularities. Note that the index
4 locus of $\Sigma_{\mathbb{P}}^{13}$ is $\{p_{3}^{2}=s_{1}u\}\subset\mathbb{P}(p_{3},u,s_{1},L_{123})$.
By the $(\mathbb{C}^{*})^{4}$-action of $\Sigma_{\mathbb{P}}^{13}$
(Proposition \ref{prop:GL}), any general point of this locus can
be moved to the point $\mathsf{q=}\{p_{3}=u=s_{1}=L_{123}\}\subset\mathbb{P}(p_{3},u,s_{1},L_{123})$.
Thus we may assume that $X$ contain the point $\mathsf{q}.$ Then
$\mathsf{q}$ is one of two $1/4(1,1,3)$-singularities. For calculations,
we perform the following coordinate change: $p_{3}'=p_{3}-s_{1},u'=-2p_{3}+s_{1}+u,v'=1/3(L_{245}s_{1})-s_{2}+v,s_{3}'=1/3(-L_{136}s_{1})+p_{2}s_{1}+s_{3},t_{2}'=-L_{124}+L_{125}-L_{135}-p_{1}-p_{4}+t_{2},\,$$L_{123}'=L_{123}-s_{1}$
replacing $p_{3},u,v,s_{3},t_{2},L_{123}$ and the other coordinates
are unchanged. Then the point $\mathsf{q}$ becomes the $s_{1}$-point
of $X$ and the tangent space of the local index 4 cover of $\Sigma_{\mathbb{P}}^{13}$
at the inverse image of $\mathsf{q}$ is defined by $\{u'=v'=s_{3}'=t_{2}'=0\}.$
The projection of $\Sigma_{\mathbb{P}}^{13}$ induced by the elimination
of $s_{1},s_{2},s_{3}'$ and its restriction to $X$ are the desired
projections.

\vspace{3pt}

\noindent \textbf{No.1766:} In this case, we seek a projection of
$X$ from one of five $1/3(1,1,2)$-singularities. Note that the index
3 locus of $\Sigma_{\mathbb{P}}^{13}$ is $\{p_{3}^{2}=s_{1}u,\,6L_{124}p_{3}^{2}-6L_{125}p_{3}u+6L_{135}u^{2}-6p_{3}s_{1}t_{2}=0,\,-9L_{125}p_{3}^{2}+9L_{124}p_{3}s_{1}+9L_{135}p_{3}u-9s_{1}^{2}t_{2}=0\}\subset\mathbb{P}(p_{3},u,s_{1},t_{2},L_{124},L_{125},L_{135})$.
By the $(\mathbb{C}^{*})^{4}$-action of $\Sigma_{\mathbb{P}}^{13}$
(Proposition \ref{prop:GL}), any general point of this locus can
be moved to the point $\mathsf{r=}\{p_{3}=u=s_{1}=L_{135},L_{124}=as_{1},L_{125}=bs_{1},t_{2}=(a-b+1)s_{1}\}\subset\mathbb{P}(p_{3},u,s_{1},t_{2},L_{124},L_{125},L_{135})$
for some $a,b\in\mathbb{C}$. Thus we may assume that $X$ contain
the point $\mathsf{r}.$ Then $\mathsf{r}$ is one of five $1/3(1,1,2)$-singularities.
For calculations, we perform the following coordinate change: 

\begin{align*}
 & p_{3}'=p_{3}-s_{1},u'=-u+2(p_{3}-s_{1})+s_{1},v'=1/3L_{245}s_{1}+p_{2}s_{1}-bp_{2}s_{1}-s_{2}+v,\\
 & s_{3}'=-(L_{136}s_{1})/3-p_{1}s_{1}+(1+a-b)p_{4}s_{1}+s_{3},\\
 & t_{2}'=-L_{124}+L_{125}-L_{135}-(1+a-2b)p_{3}-u+(2+a-2b)s_{1}+t_{2},\\
 & L_{123}'=L_{123}-ds_{2},L_{124}'=L_{124}-as_{1},L_{125}'=L_{125}-bs_{1},L_{135}'=L_{135}-s_{1},
\end{align*}

replacing $p_{3},u,v,s_{3},t_{2},L_{123,}L_{124},L_{125},L_{135}$,
and the other coordinates are unchanged. Then the point r becomes
the $s_{1}$-point of $X$ and the tangent space of the local index
4 cover of $\Sigma_{\mathbb{P}}^{13}$ at the inverse image of $\mathsf{r}$
is defined by $\{u'=v'=s_{3}'=t_{2}'=0\}.$ The projection of $\Sigma_{\mathbb{P}}^{13}$
induced by the elimination of $s_{1},s_{2},s_{3}'$ and its restriction
to $X$ are the desired projections (we note that we introduce the
new coordinates $L_{123}'$ to make the restriction of the elimination
of $s_{1},s_{2},s_{3}'$ for $\Pi_{\mathbb{P}}^{13}$ to $X$ coincides
with the elimination of $s_{1},s_{2},s_{3}'$ for $X$).

Thus we have shown Theorem \ref{thm:main} (1-3).

\section{Affine variety $H_{\mathbb{A}}^{13}$\label{sec:Affine-variety}}

\subsection{Hypermatrix equation of $H_{\mathbb{A}}^{13}$}

Let $V_{1}$, $V_{2}$, $V_{3}$ are three $2$-dimensional vector
spaces. Let $\langle\bm{e}_{1},\bm{e}_{2}\rangle$, $\langle\bm{f}_{1},\bm{f}_{2}\rangle$,
and $\langle\bm{g}_{1},\bm{g}_{2}\rangle$ are their bases, respectively.
We denote by 
\[
\bm{x}_{1}=\begin{pmatrix}x_{11}\\
x_{21}
\end{pmatrix},\quad\bm{x}_{2}=\begin{pmatrix}x_{12}\\
x_{22}
\end{pmatrix},\quad\bm{x}_{3}=\begin{pmatrix}x_{13}\\
x_{23}
\end{pmatrix}
\]
their coordinate vectors, respectively. We also denote by $p_{ijk}$
the coordinate of $\bm{e}_{i}\otimes\bm{f}_{j}\otimes\bm{g}_{k}$
($1\leq i,j,k\leq2$) in $V_{1}\otimes V_{2}\otimes V_{3}$.

\vspace{5pt}

We consider alternating inner products for $V_{1}$, $V_{2}$, $V_{3}$
satisfying $\bm{e}_{i}\cdot\bm{e}_{i}=0$ for $i=1,2$, $\bm{e}_{1}\cdot\bm{e}_{2}=1$
and $\bm{e}_{2}\cdot\bm{e}_{1}=-1$ and similarly for the bases of
$V_{2}$, $V_{3}$.

Let 
\[
c_{12}\colon(V_{1}\otimes V_{2}\otimes V_{3})\times V_{1}\times V_{2}\to V_{3}
\]
be the contraction map with respect to the above defined inner products,
and $c_{23},c_{31}$ similarly defined. We write down the calculation
of the contraction maps as follows. It is convenient to set

{\small{}
\begin{align*}
 & D_{11}^{(1)}:=\begin{vmatrix}p_{111} & x_{11}\\
p_{211} & x_{21}
\end{vmatrix},D_{12}^{(1)}:=\begin{vmatrix}p_{112} & x_{11}\\
p_{212} & x_{21}
\end{vmatrix},\,D_{21}^{(1)}:=\begin{vmatrix}p_{121} & x_{11}\\
p_{221} & x_{21}
\end{vmatrix},\,D_{22}^{(1)}:=\begin{vmatrix}p_{122} & x_{11}\\
p_{222} & x_{21}
\end{vmatrix},\\
 & D_{11}^{(2)}:=\begin{vmatrix}p_{111} & x_{12}\\
p_{121} & x_{22}
\end{vmatrix},\,D_{12}^{(2)}:=\begin{vmatrix}p_{112} & x_{12}\\
p_{122} & x_{22}
\end{vmatrix},\,D_{21}^{(2)}:=\begin{vmatrix}p_{211} & x_{12}\\
p_{221} & x_{22}
\end{vmatrix},D_{22}^{(2)}:=\begin{vmatrix}p_{212} & x_{12}\\
p_{222} & x_{22}
\end{vmatrix},\\
 & D_{11}^{(3)}:=\begin{vmatrix}p_{111} & x_{13}\\
p_{112} & x_{23}
\end{vmatrix},\,D_{12}^{(3)}:=\begin{vmatrix}p_{121} & x_{13}\\
p_{122} & x_{23}
\end{vmatrix},\,D_{21}^{(3)}:=\begin{vmatrix}p_{211} & x_{13}\\
p_{212} & x_{23}
\end{vmatrix},D_{22}^{(3)}:=\begin{vmatrix}p_{221} & x_{13}\\
p_{222} & x_{23}
\end{vmatrix}.
\end{align*}
}{\small\par}
\begin{prop}
For $\bm{x}_{1}=x_{11}\bm{e}_{1}+x_{21}\bm{e}_{2}\in V_{1}$, $\bm{x}_{2}=x_{12}\bm{f}_{1}+x_{22}\bm{f}_{2}\in V_{2}$,
$\bm{x}_{3}=x_{13}\bm{g}_{1}+x_{23}\bm{g}_{2}\in V_{3}$, and $P={\displaystyle \sum_{1\leq i,j,k\leq2}p_{ijk}\bm{e}_{i}\otimes\bm{f}_{j}\otimes\bm{g}_{k}\in V_{1}\otimes V_{2}\otimes V_{3}}$,
the evaluations $c_{12}(P,\bm{x}_{1},\bm{x}_{2})$, $c_{23}(P,\bm{x}_{2},\bm{x}_{3})$,
$c_{31}(P,\bm{x}_{3},\bm{x}_{1})$ of the contraction maps are written
down as follows, where by abuse of notation, we denote by the same
symbols the coordinate vectors for $c_{12}(P,\bm{x}_{1},\bm{x}_{2})$,
$c_{23}(P,\bm{x}_{2},\bm{x}_{3})$, $c_{31}(P,\bm{x}_{3},\bm{x}_{1})$:
{\small{}
\begin{align*}
c_{12}(P,\bm{x}_{1},\bm{x}_{2})=\begin{pmatrix}-D_{21}^{(1)} & D_{11}^{(1)}\\
-D_{22}^{(1)} & D_{12}^{(1)}
\end{pmatrix}\bm{x}_{2}=\begin{pmatrix}-D_{21}^{(2)} & D_{11}^{(2)}\\
-D_{22}^{(2)} & D_{12}^{(2)}
\end{pmatrix}\bm{x}_{1},\\
c_{23}(P,\bm{x}_{2},\bm{x}_{3})=\begin{pmatrix}-D_{12}^{(2)} & D_{11}^{(2)}\\
-D_{22}^{(2)} & D_{21}^{(2)}
\end{pmatrix}\bm{x}_{3}=\begin{pmatrix}-D_{12}^{(3)} & D_{11}^{(3)}\\
-D_{22}^{(3)} & D_{21}^{(3)}
\end{pmatrix}\bm{x}_{2},\\
c_{31}(P,\bm{x}_{3},\bm{x}_{1})=\begin{pmatrix}-D_{21}^{(3)} & D_{11}^{(3)}\\
-D_{22}^{(3)} & D_{12}^{(3)}
\end{pmatrix}\bm{x}_{1}=\begin{pmatrix}-D_{12}^{(1)} & D_{11}^{(1)}\\
-D_{22}^{(1)} & D_{21}^{(1)}
\end{pmatrix}\bm{x}_{3}.
\end{align*}
}{\small\par}
\end{prop}

\begin{proof}
The claim follows from straightforward calculations. For example,
\begin{align*}
 & c_{12}(P,\bm{x}_{1},\bm{x}_{2})=\\
 & (p_{111}x_{21}x_{22}-p_{121}x_{21}x_{12}-p_{211}x_{11}x_{22}+p_{221}x_{11}x_{12})\bm{g}_{1}\\
 & +(p_{112}x_{21}x_{22}-p_{122}x_{21}x_{12}-p_{212}x_{11}x_{22}+p_{222}x_{11}x_{12})\bm{g}_{2}.
\end{align*}
It is easy to see this can be seen in the two ways stated as above. 
\end{proof}
Based on the above considerations, we introduce the following affine
variety $H_{\mathbb{A}}^{13}$:
\begin{defn}
We set 
\[
V_{H}:=V_{1}\oplus V_{2}\oplus V_{3}\oplus V_{1}\otimes V_{2}\otimes V_{3}\oplus\mathbb{C}^{\oplus3}.
\]
In the affine space $\mathbb{A}^{17}(V_{H})$ with coordinates $\bm{x}_{1},\bm{x}_{2},\bm{x}_{3},P,u_{1},u_{2},u_{3}$,
we define the affine variety $H_{\mathbb{A}}^{13}$ by the following
set of equations: 

{\small{}
\begin{align*}
 & \mathsf{G}_{1}=u_{1}\bm{x}_{1}-\begin{pmatrix}-D_{12}^{(3)} & D_{11}^{(3)}\\
-D_{22}^{(3)} & D_{21}^{(3)}
\end{pmatrix}\bm{x}_{2}=\bm{0},\\
 & \mathsf{G}_{2}=u_{2}\bm{x}_{2}-\begin{pmatrix}-D_{12}^{(1)} & D_{11}^{(1)}\\
-D_{22}^{(1)} & D_{21}^{(1)}
\end{pmatrix}\bm{x}_{3}=\bm{0},\\
 & \mathsf{G}_{3}=u_{3}\bm{x}_{3}-\begin{pmatrix}-D_{21}^{(2)} & D_{11}^{(2)}\\
-D_{22}^{(2)} & D_{12}^{(2)}
\end{pmatrix}\bm{x}_{1}=\bm{0},\\
 & G_{4}=u_{1}u_{2}-(D_{12}^{(3)}D_{21}^{(3)}-D_{11}^{(3)}D_{22}^{(3)})=0,\\
 & G_{5}=u_{2}u_{3}-(D_{12}^{(1)}D_{21}^{(1)}-D_{11}^{(1)}D_{22}^{(1)})=0,\\
 & G_{6}=u_{3}u_{1}-(D_{12}^{(2)}D_{21}^{(2)}-D_{11}^{(2)}D_{22}^{(2)})=0.
\end{align*}
}{\small\par}
\end{defn}

\subsection{${\rm (GL_{2})^{3}\times S_{3}}$-action on $H_{\mathbb{A}}^{13}$}

In this subsection, we show the following:
\begin{prop}
\label{prop:actionH}$H_{\mathbb{A}}^{13}$ has a $({\rm GL}_{2})^{3}\times S_{3}$-action.
\end{prop}

\begin{proof}
There exists three ${\rm GL_{2}}$-actions on $H_{\mathbb{A}}^{13}$
defined by the following rules (1)--(3):\vspace{3pt}

\noindent(1) For an element $g_{1}\in{\rm GL_{2}}$, we set $\bm{x}_{1}\mapsto g_{1}\bm{x}_{1}$,
$u_{2}\mapsto(\det g_{1})u_{2}$, $u_{3}\mapsto(\det g_{1})u_{3}$,
{\small{}
\[
\left(\begin{array}{cccc}
p_{111} & p_{112} & p_{121} & p_{122}\\
p_{211} & p_{212} & p_{221} & p_{222}
\end{array}\right)\mapsto g_{1}\left(\begin{array}{cccc}
p_{111} & p_{112} & p_{121} & p_{122}\\
p_{211} & p_{212} & p_{221} & p_{222}
\end{array}\right),
\]
}and other coordinates being unchanged. We can check that $\mathsf{G}_{1}\mapsto g_{1}\mathsf{G}_{1}$,
$\mathsf{G_{2}}\mapsto(\det g_{1})\mathsf{G_{2}}$, $\mathsf{G}_{3}\mapsto(\det g_{1})\mathsf{G}_{3}$,
$G_{4}\mapsto(\det g_{1})G_{4}$, $G_{5}\mapsto(\det g_{1})^{2}G_{5}$,
$G_{6}\mapsto(\det g_{1})G_{6}$.

\vspace{3pt}

\noindent(2) For an element $g_{2}\in{\rm GL_{2}}$, we set $\bm{x}_{2}\mapsto g_{2}\bm{x}_{2}$,
$u_{1}\mapsto(\det g_{2})u_{1}$, $u_{3}\mapsto(\det g_{2})u_{3}$,
{\small{}
\[
\left(\begin{array}{cccc}
p_{111} & p_{112} & p_{211} & p_{212}\\
p_{121} & p_{122} & p_{221} & p_{222}
\end{array}\right)\mapsto g_{2}\left(\begin{array}{cccc}
p_{111} & p_{112} & p_{211} & p_{212}\\
p_{121} & p_{122} & p_{221} & p_{222}
\end{array}\right),
\]
} and other coordinates being unchanged. The changes of the equations
of $H_{\mathbb{A}}^{13}$ are similar to the case (1), so we omit
them.

\vspace{3pt}

\noindent(3) For and element $g_{3}\in{\rm GL_{2}}$, we set $\bm{x}_{3}\mapsto g_{3}\bm{x}_{3}$,
$u_{1}\mapsto(\det g_{3})u_{1}$, $u_{2}\mapsto(\det g_{3})u_{2}$,
{\small{}
\[
\left(\begin{array}{cccc}
p_{111} & p_{121} & p_{211} & p_{221}\\
p_{112} & p_{122} & p_{212} & p_{222}
\end{array}\right)\mapsto g_{3}\left(\begin{array}{cccc}
p_{111} & p_{121} & p_{211} & p_{221}\\
p_{112} & p_{122} & p_{212} & p_{222}
\end{array}\right),
\]
}and other coordinates being unchanged. The changes of the equations
of $H_{\mathbb{A}}^{13}$ are similar to the case (1), so we omit
them.

We can also check that these three ${\rm GL_{2}}$-actions on $H_{\mathbb{A}}^{13}$
mutually commute. Thus these actions induce a $({\rm GL_{2})^{3}}$-action
on $H_{\mathbb{A}}^{13}$.

Moreover we may define an $S_{3}$-action by associating the three
numbers $1,2,3$ to the three coordinates $u_{1},u_{2},u_{3}$, the
three vectors $\bm{x}_{1},\bm{x}_{2},\bm{x}_{3}$, and the three subscripts
$i,j,k$ for $p_{ijk}$. For example, by $(12)\in S_{3}$, $u_{1}$
and $u_{2}$, $\bm{x}_{1}$ and $\bm{x}_{2}$, $p_{ijk}$ and $p_{jik}$
are interchanged respectively.

Hence $H_{\mathbb{A}}^{13}$ has a $({\rm GL}_{2})^{3}\times S_{3}$-action.
\end{proof}

\subsection{$\mathbb{P}^{2}\times\mathbb{P}^{2}$-fibration associated to $H_{\mathbb{A}}^{13}$}

\label{fib13} We denote by $\mathbb{A}(P)$ the affine space with
coordinates $p_{ijk}$~$(1\leq i,j,k\leq2)$. Note that any equation
of $\text{\ensuremath{H_{\mathbb{A}}^{13}}}$ is of degree two if
we regard the coordinates of $\mathbb{A}(P)$ as constants. Therefore,
considering the coordinates in the equations of $\text{\text{\ensuremath{H_{\mathbb{A}}^{13}}}}$
except those of $\mathbb{A}(P)$ as projective coordinates, we obtain
a variety in $\mathbb{A}(P)\times\mathbb{P}(\bm{x}_{1},\bm{x}_{2},\bm{x}_{3},u_{1},u_{2,}u_{3})$.
We denote this variety by $\hat{H}$. 
\begin{defn}
We set{\small{} $\mathsf{S}=\left(\begin{array}{ccc}
s_{11} & s_{12} & s_{13}\\
s_{12} & s_{22} & s_{23}\\
s_{13} & s_{23} & s_{33}
\end{array}\right)$} and {\small{}$\bm{{\sigma}}=\left(\begin{array}{c}
\sigma_{1}\\
\sigma_{2}\\
\sigma_{3}
\end{array}\right)$}, and denote by $\mathsf{S}^{\dagger}$ the adjoint matrix of $\mathsf{S}$.
We define 
\[
\mathbb{P}^{2,2}:=\{(\mathsf{S},\bm{\sigma})\mid\mathsf{S}^{\dagger}=O,\mathsf{S}\bm{\sigma}=\bm{0}\}\subset\mathbb{P}(\mathsf{S},\bm{\sigma}),
\]
which is introduced in \cite{key-14} as a degeneration of $\mathbb{P}^{2}\times\mathbb{P}^{2}$
and is shown to be a singular del Pezzo 4-fold of degree six (see
also \cite{key-15}). 
\end{defn}

We also denote by $\mathcal{H}$ the Cayley hyperdeterminant. Explicitly
we have {\small{}
\begin{align*}
\mathcal{H}= & p_{111}^{2}p_{222}^{2}+p_{112}^{2}p_{221}^{2}+p_{121}^{2}p_{212}^{2}+p_{122}^{2}p_{211}^{2}\\
 & -2p_{111}p_{122}p_{211}p_{222}-2p_{111}p_{121}p_{212}p_{222}-2p_{111}p_{112}p_{221}p_{222}\\
 & -2p_{121}p_{122}p_{211}p_{212}-2p_{112}p_{122}p_{211}p_{221}-2p_{112}p_{121}p_{212}p_{221}\\
 & +4p_{111}p_{122}p_{212}p_{221}+4p_{112}p_{121}p_{211}p_{222}.
\end{align*}
}{\small\par}

It is classically known that $\{\mathcal{H}=0\}$ is stable under
the $({\rm GL}_{2})^{3}\times S_{3}$-action on the affine space $\mathbb{A}(P)$
defined as in Proposition \ref{prop:actionH}.
\begin{prop}
Let $\rho_{H}\colon\hat{H}\to\mathbb{A}(P)$ be the natural projection.
The $\rho_{H}$-fibers over points of the complement of $\{\mathcal{H}=0\}$
are isomorphic to $\mathbb{P}^{2}\times\mathbb{P}^{2}$, and the $\rho_{H}$-fibers
over points of the nonempty ${\rm (GL_{2})^{3}\times S_{3}}$ -stable
open subset of $\{\mathcal{H}=0\}$ are isomorphic to $\mathbb{P}^{2,2}$.
\end{prop}

\begin{proof}
We give descriptions of all the $\rho_{H}$-fibers. Note that the
${\rm (GL_{2})^{3}\times S_{3}}$-action on $\text{\text{\ensuremath{H_{\mathbb{A}}^{13}}}}$
induces that on $\hat{H}$. The induced $({\rm SL}_{2})^{3}$-action
on $\mathbb{A}(P)$ is studied in \cite[Chap.14, Ex.4.5]{key-12};
$\mathbb{A}(P)$ has the seven orbits. By this result and by considering
the induced $S_{3}$-action on $\mathbb{A}(P)$ together, we see that
$\mathbb{A}(P)$ has five ${\rm (GL_{2})^{3}\times S_{3}}$-orbits.
To state the result, we set $\mathsf{p}_{1}:=$ the $p_{111}$-point,
$\mathsf{p}_{2}$:= the point with $p_{111}=p_{221}=1$ and the other
coordinates being zero, $\mathsf{p}_{3}:=$ the point with $p_{111}=p_{122}=p_{212}=1$
and the other coordinates being zero, $\mathsf{p}_{4}:=$ the point
with $p_{111}=p_{222}=1$ and the other coordinates being zero. Then
the five ${\rm (GL_{2})^{3}\times S_{3}}$-orbits are the following:

\vspace{3pt}

\noindent(1) The origin of $\mathbb{A}(P)$. 

\vspace{3pt}

\noindent(2) The $4$-dimensional orbit of $\mathsf{p}_{1}$, which
is the complement of the origin in the affine cone over the Segre
variety $\mathbb{P}^{1}\times\mathbb{P}^{1}\times\mathbb{P}^{1}$.

\vspace{3pt}

\noindent(3) The $5$-dimensional orbit of $\mathsf{p}_{2}$, which
is the complement of the affine cone over the Segre variety $\mathbb{P}^{1}\times\mathbb{P}^{1}\times\mathbb{P}^{1}$
in the union of three copies of the affine cone over the Segre variety
$\mathbb{P}^{1}\times\mathbb{P}^{3}$.

\vspace{3pt}

\noindent(4) The $7$-dimensional orbit of $\mathsf{p}_{3}$, which
is the complement of the union of three copies of the affine cone
over the Segre variety $\mathbb{P}^{1}\times\mathbb{P}^{3}$ in $\{\mathcal{H}=0\}$.

\vspace{3pt}

\noindent(5) The open orbit of $\mathsf{p}_{4}$, which is the complement
of $\{\mathcal{H}=0\}$ in $\mathbb{A}(P)$.

\vspace{3pt}

Then $\rho_{H}$-fibers are isomorphic to the fibers over the origin
or the $\mathsf{p}_{i}$-points for some $i=1,2,3,4$, which are easily
described by the equations of $H_{\mathbb{A}}^{13}$ as follows:

\vspace{3pt}

\noindent(a) The fiber over the origin is the union of three $\mathbb{P}^{4}$'s
and a $\mathbb{P}^{5}$;{\small{}
\begin{align*}
 & \{u_{1}=u_{2}=u_{3}=0\}\\
 & \cup\{u_{1}=u_{2}=x_{13}=x_{23}=0\}\cup\{u_{1}=u_{3}=x_{12}=x_{22}=0\}\cup\{u_{2}=u_{3}=x_{11}=x_{21}=0\}.
\end{align*}
}{\small\par}

\vspace{3pt}

\noindent(b) The fiber over the point $\mathsf{p}_{1}$ is the union
of three quadric $4$-folds of rank $4$; {\small{}
\begin{align*}
 & \{u_{1}=u_{2}=x_{23}=u_{3}x_{13}-x_{21}x_{22}=0\}\cup\{u_{1}=u_{3}=x_{22}=u_{2}x_{12}-x_{23}x_{21}=0\}\\
 & \cup\{u_{2}=u_{3}=x_{21}=u_{1}x_{11}-x_{23}x_{22}=0\}.
\end{align*}
}{\small\par}

\vspace{3pt}

\noindent(c) The fiber over the point $\mathsf{p}_{2}$ is the union
of a quadric $4$-fold of rank $6$ and the cone over a hyperplane
section of $\mathbb{P}^{1}\times\mathbb{P}^{3}$; {\small{}
\begin{align*}
 & \{u_{1}=u_{2}=x_{23}=u_{3}x_{13}-x_{11}x_{12}-x_{21}x_{22}=0\}\\
 & \cup\left\{ u_{3}=0,{\rm rank\,}\begin{pmatrix}x_{23} & u_{2} & x_{11} & x_{21}\\
u_{1} & -x_{23} & x_{22} & -x_{12}
\end{pmatrix}\leq1\right\} .
\end{align*}
}{\small\par}

\vspace{3pt}

\noindent(d) The fiber over the point $\mathsf{p}_{3}$: {\small{}
\[
\left\{ {\rm rank}\,\begin{pmatrix}u_{1} & x_{13} & x_{22}\\
x_{13} & u_{2} & -x_{21}\\
x_{22} & -x_{21} & -u_{3}
\end{pmatrix}\leq1,\begin{pmatrix}u_{1} & x_{13} & x_{22}\\
x_{13} & u_{2} & -x_{21}\\
x_{22} & -x_{21} & -u_{3}
\end{pmatrix}\begin{pmatrix}-x_{11}\\
x_{12}\\
x_{23}
\end{pmatrix}=\bm{o}\right\} ,
\]
}which is isomorphic to $\mathbb{P}^{2,2}$. 

\vspace{3pt}

\noindent(e) The fiber over the point $\mathsf{p}_{4}$: {\small{}
\[
\left\{ {\rm rank}\,\begin{pmatrix}u_{1} & x_{13} & x_{22}\\
x_{23} & u_{2} & x_{11}\\
x_{12} & x_{21} & u_{3}
\end{pmatrix}\leq1\right\} ,
\]
}which is isomorphic to $\mathbb{P}^{2}\times\mathbb{P}^{2}$. 

Therefore we have shown the assertions.
\end{proof}

\subsection{Relation between $\Pi_{\mathbb{A}}^{15}$ and $H_{\mathbb{A}}^{13}$. }

We establish the following relation between $\Pi_{\mathbb{A}}^{15}$
and $H_{\mathbb{A}}^{13}$. We define some notation. Let $\mathbb{A}_{\Pi}^{10}$
be the affine $10$-space as in the section \ref{sec:Singular-locus-Pi}
(but we do not consider here that $\mathbb{A}_{\Pi}^{10}\subset\Pi_{\mathbb{A}}^{15}$).
Let $\mathbb{A}(L)$ be the affine space with the coordinates $L_{123}$,
$L_{124}$, $L_{125}$, $L_{126}$, $L_{135}$, $L_{136}$, $L_{245}$,
$L_{246}$ and $\mathbb{A}(A,B)$ the affine space with the coordinates
$A$, $B$. We set $\widetilde{\mathbb{A}}_{\Pi}^{10}=\mathbb{A}(A,B)\times\mathbb{A}(L)$.
Let $b\colon\widetilde{\mathbb{A}}_{\Pi}^{10}\to\mathbb{A}_{\Pi}^{10}$
be the morphism defined with 
\[
t_{1}=-A^{2}-AB-B^{2},\,t_{2}=-AB(A+B)
\]
 and the remaining coordinates being unchanged. The morphism $b$
is a finite morphism of degree six since $b$ is identified with the
morphism $\{A+B+C=0\}\times\mathbb{A}(L)\to\mathbb{A}_{\Pi}^{10}$
defined with $t_{1}=AB+BC+CA$, $t_{2}=ABC$ and the remaining coordinates
being unchanged, where $\{A+B+C=0\}$ is a closed subset of the affine
$3$-space with the coordinates $A,B,C$. Let $\widetilde{\Pi}_{\mathbb{A}}^{15}:=\Pi_{\mathbb{A}}^{15}\times_{\mathbb{A}_{\Pi}^{10}}\widetilde{\mathbb{A}}_{\Pi}^{10}$,
and $b_{\Pi}\colon\widetilde{\Pi}_{\mathbb{A}}^{15}\to\Pi_{\mathbb{A}}^{15}$
and $p_{\Pi}\colon\widetilde{\Pi}_{\mathbb{A}}^{15}\to\widetilde{\mathbb{A}}_{\Pi}^{10}$
the naturally induced morphisms. Let $p_{H}\colon H_{\mathbb{A}}^{13}\to\mathbb{A}(P)$
be the natural projection. Let $U_{AB}$ be the open subset $\{(A-B)(2A+B)(A+2B)\not=0\}$
of $\mathbb{A}(A,B)$. Note that the morphism $b$ is unramified on
$U_{AB}\times\mathbb{A}(L)\subset\widetilde{\mathbb{A}}_{\Pi}^{10}$.
\begin{prop}
\label{prop:isom}There exists an isomorphism from $U_{AB}\times\mathbb{A}(L)$
to $U_{AB}\times\mathbb{A}(P)$, and an isomorphism from $\widetilde{\Pi}_{\mathbb{A}}^{15}\cap\{(A-B)(2A+B)(A+2B)\not=0\}$
to $U_{AB}\times H_{\mathbb{A}}^{13}$ which fit into the following
commutative diagram:{\small{}
\[
\xymatrix{\widetilde{\Pi}_{\mathbb{A}}^{15}\cap\{(A-B)(2A+B)(A+2B)\not=0\}\ar[r]\ar[d]_{p_{\Pi}} & U_{AB}\times H_{\mathbb{A}}^{13}\ar[d]^{p_{H}\times{\rm id}}\\
U_{AB}\times\mathbb{A}(L)\ar[r] & U_{AB}\times\mathbb{A}(P)
}
\]
}{\small\par}
\end{prop}

\begin{proof}
We can directly check that a morphism from $U_{AB}\times\mathbb{A}(L)$
to $U_{AB}\times\mathbb{A}(P)$, and a morphism from $\widetilde{\Pi}_{\mathbb{A}}^{15}\cap\{(A-B)(2A+B)(A+2B)\not=0\}$
to $U_{AB}\times H_{\mathbb{A}}^{13}$ can be defined by the following
equalities and they are actually isomorphisms with the desired properties.

{\small{}
\begin{align*}
p_{111} & =-3L_{246},\\
p_{112} & =3L_{245}+3(L_{124}+(A+B)L_{126})/((2A+B)(A+2B)),\\
p_{121} & =3L_{245}+3(-L_{124}+BL_{126})/((A-B)(A+2B)),\\
p_{122} & =-3L_{136}+3(L_{123}-AL_{125})/((A-B)(2A+B)),\\
p_{211} & =3L_{245}+3(L_{124}-AL_{126})/((A-B)(2A+B)),\\
p_{212} & =-3L_{136}+3(-L_{123}+BL_{125})/((A-B)(A+2B)),\\
p_{221} & =-3L_{136}+3(L_{123}+(A+B)L_{125})/((2A+B)(A+2B)),\\
p_{222} & =3L_{135},
\end{align*}
\begin{align*}
 & x_{11}=-Ap_{1}-A^{2}p_{2}+u,\,x_{21}=-Ap_{3}-A^{2}p_{4}+v,\\
 & x_{12}=Bp_{1}+B^{2}p_{2}-u,\,x_{22}=Bp_{3}+B^{2}p_{4}-v,\\
 & x_{13}=-(A+B)p_{1}+(A+B)^{2}p_{2}-u,\,x_{23}=-(A+B)p_{3}+(A+B)^{2}p_{4}-v,
\end{align*}
}{\small\par}

{\small{}
\begin{align*}
 & u_{1}=3((AB+B^{2})s_{1}-As_{2}-s_{3})+3/((A-B)(2A+B))\big((2AL_{123}-(AB+B^{2})L_{125})p_{1}\\
 & -((2A^{2}+AB+B^{2})L_{123}-2(A^{2}B+AB^{2})L_{125})p_{2}+(2AL_{124}-(AB+B^{2})L_{126})p_{3}\\
 & -((2A^{2}+AB+B^{2})L_{124}-2(A^{2}B+AB^{2})L_{126})p_{4}+(L_{123}-AL_{125})u+(L_{124}-AL_{126})v)\big),
\end{align*}
\begin{align*}
 & u_{2}=3((A^{2}+AB)s_{1}-Bs_{2}-s_{3})+3/((A-B)(A+2B))\big((-2BL_{123}+(A^{2}+AB)L_{125})p_{1}\\
 & +((A^{2}+AB+2B^{2})L_{123}-2(A^{2}B+AB^{2})L_{125})p_{2}+(-2BL_{124}+(A^{2}+AB)L_{126})p_{3}\\
 & +((A^{2}+AB+2B^{2})L_{124}-2(A^{2}B+AB^{2})L_{126})p_{4}+(-L_{123}+BL_{125})u+(-L_{124}+BL_{126})v)\big),
\end{align*}
\begin{align*}
 & u_{3}=3(ABs_{1}-(A+B)s_{2}+s_{3})+3/((2A+B)(A+2B))\big((2(A+B)L_{123}-ABL_{125})p_{1}\\
 & +((2A^{2}+3AB+2B^{2})L_{123}-2(A^{2}B+AB^{2})L_{125})p_{2}+(2(A+B)L_{124}-ABL_{126})p_{3}\\
 & +((2A^{2}+3AB+2B^{2})L_{124}-2(A^{2}B+AB^{2})L_{126})p_{4}\\
 & +(-L_{123}-(A+B)L_{125})u+(-L_{124}-(A+B)L_{126})v)\big).
\end{align*}
}{\small\par}
\end{proof}

\section{More on geometry of $\Pi_{\mathbb{A}}^{15}$ \label{sec:More-on-geometry}}

\subsection{$\mathbb{P}^{2}\times\mathbb{P}^{2}$-fibration associated to $\Pi_{\mathbb{A}}^{15}$}

Note that the equations of $\Pi_{\mathbb{A}}^{15}$ are of degree
two when we consider the coordinates of $\mathbb{A}_{\Pi}^{10}$ are
constants. Thus they define a subvariety $\hat{\Pi}$ of $\mathbb{P}(p_{1},p_{2},p_{3},p_{4},u,v,s_{1},s_{2},s_{3})\times\mathbb{A}_{\Pi}^{10}$,
which is similar to $\hat{H}$ for $H_{\mathbb{\mathbb{A}}}^{13}$.
We note that the image of $U_{AB}\times\mathbb{A}(L)$ by the morphism
$b$ is the open subset $\{4t_{1}^{3}+27t_{2}^{2}\not=0\}\times\mathbb{A}(L)$
of $\mathbb{A}_{\Pi}^{10}$. Let $\{\mathcal{H}'=0\}$ be the image
by $b$ of the pull-back on $U_{AB}\times\mathbb{A}(L)$ of $U_{AB}\times\{\mathcal{H}=0\}$. 
\begin{prop}
\label{prop:P2P2Pi}Let $\rho_{\Pi}\colon\hat{\Pi}\to\mathbb{A}_{\Pi}^{10}$
be the natural projection. The $\rho_{\Pi}$-fibers over points of
the complement of $\{\mathcal{H}'=0\}$ in $\{4t_{1}^{3}+27t_{2}^{2}\not=0\}\times\mathbb{A}(L)$
are isomorphic to $\mathbb{P}^{2}\times\mathbb{P}^{2}$, and the $\rho_{\Pi}$-fibers
over points of a nonempty subset of $\{\mathcal{H}'=0\}$ are isomorphic
to $\mathbb{P}^{2,2}$.
\end{prop}

\begin{proof}
Let $\hat{\Pi}'$ be the base change of $\hat{\Pi}$ over $\widetilde{\mathbb{A}}_{\Pi}^{10}$.
We have only to describe the corresponding fibers of $\hat{\Pi}'\to\widetilde{\mathbb{A}}_{\Pi}^{10}$.
Since the isomorphism from $\widetilde{\Pi}_{\mathbb{A}}^{15}\cap\{(A-B)(2A+B)(A+2B)\not=0\}$
to $H_{\mathbb{A}}^{13}\times U_{AB}$ as in Proposition \ref{prop:isom}
is linear with respect to the coordinates which are vertical with
respect to the fibrations $\rho_{\Pi}$ and $\rho_{H}$, it descends
to an isomorphism from $\hat{\Pi}'\cap\{(A-B)(2A+B)(A+2B)\not=0\}$
to the open subset $U_{AB}\times\hat{H}$. Thus we have the assertion.
\end{proof}

\subsection{More on singularities of $\Pi_{\mathbb{A}}^{15}$}

Using Proposition \ref{prop:SingPi}, we have the following by the
proof of \cite[Prop.~8.4]{key-5}:
\begin{prop}
\label{prop:MoreSing}The variety $\Pi_{\mathbb{A}}^{15}$ has only
terminal singularities with the following descriptions:

$(1)$ $\Pi_{\mathbb{A}}^{15}$ has $c(G(2,5))$-singularities along
a locally closed subset $S$ of codimension seven in $\Pi_{\mathbb{A}}^{15}$.

$(2)$ There exists a primitive $K$-negative divisorial extraction
$f\colon\tilde{\Pi}\to\Pi_{\mathbb{A}}^{15}$ such that

$({\rm i})$ singularities of $\tilde{\Pi}$ are only $c(G(2,5))$-singularities
along the strict transform of the closure of $S$, and

$({\rm ii})$ for the $f$-exceptional divisor $E_{\Pi}$, the morphism
$f|_{E_{\Pi}}$ can be identified with $\rho_{\Pi}\colon\hat{\Pi}\to\mathbb{A}_{\Pi}^{10}$. 
\end{prop}

{\small{}~}{\small\par}

\selectlanguage{american}%
{\small{}Hiromichi Takagi}{\small\par}

{\small{}Department of Mathematics, Gakushuin University, }{\small\par}

{\small{}Mejiro, Toshima-ku, Tokyo 171-8588, Japan }{\small\par}

{\small{}e-mail: hiromici@math.gakushuin.ac.jp}{\small\par}

\selectlanguage{english}%
\end{document}